\documentclass{article}
  
\usepackage[utf8]{inputenc}
\usepackage{graphicx}
\usepackage{subfigure}
\usepackage{amsmath}
\usepackage{amssymb}
\usepackage{amsthm}
\usepackage{dsfont}
\usepackage{geometry}
\usepackage{comment}
\usepackage{url,hyperref}
\newcommand{\reach}{{\rm rch}}
\newcommand{\hf}{\hat{f}}
\newcommand{\hfpl}{\hat{f}_{\rm PG}}
\newcommand{\fpl}{f_{\rm PG}}

\newcommand{\im}{{\rm im}}
\newcommand{\res}{{\rm res}}

\newcommand{\fat}{{\rm fat}}
\newcommand{\X}{\mathcal{X}}
\newcommand{\Z}{\mathcal{Z}}
\newcommand{\hX}{\mathbb X}
\newcommand{\hZ}{\mathbb Z}
\newcommand{\U}{\mathcal{U}}
\newcommand{\V}{\mathcal{V}}
\newcommand{\R}{\mathbb{R}}
\newcommand{\N}{\mathbb{N}}
\newcommand{\E}{\mathbb{E}}
\newcommand{\proba}{\mathbb{P}}
\newcommand{\mapper}{\mathrm{M}}
\newcommand{\reeb}{\mathrm{R}}

\newcommand{\Hist}{{\rm Hist}}

\newcommand{\Gn}{G_{\delta_n}}
\newcommand{\egn}{d_{{\rm H},n}}
\newcommand{\deh}{d^E_{{\rm H}}}

\newcommand{\esd}{p}

\usepackage{tikz}
\usetikzlibrary{arrows}
\usetikzlibrary{matrix}
\usepackage{tikz-cd}

\newtheorem{theorem}{Theorem}[section]
\newtheorem{proposition}[theorem]{Proposition}
\newtheorem{lemma}[theorem]{Lemma}
\newtheorem{definition}[theorem]{Definition}

\title{Statistical analysis of Mapper for stochastic and multivariate filters}

\author{Mathieu Carri\`ere\footnote{DataShape, Inria Sophia Antipolis, Biot, France, {\tt mathieu.carriere@inria.fr}}, Bertrand Michel\footnote{Laboratoire de Math\'ematiques 
Jean Leray, UMR\_C 6629, Ecole Centrale de Nantes,  France, {\tt bertrand.michel@ec-nantes.fr}}}

\begin{document}
\maketitle

\begin{abstract}
Reeb spaces, as well as their discretized versions called Mappers, are common descriptors used in 
Topological Data Analysis, with plenty of applications in various fields of science, such as computational 
biology and data visualization, among others. The stability and quantification of the rate of convergence 
of the Mapper to the Reeb space has been studied a lot in recent 
works~\cite{Brown2019, Carriere2017, Carriere2018a, Munch2016}, focusing on the case where a 
scalar-valued filter is used for the computation of Mapper. On the other hand, much less is known in 
the multivariate case, when the codomain of the filter is $\R^\esd$, and in the general case, when it is 
a general metric space $(\Z,d_\Z)$, instead of $\R$. The few 
results that are available in this setting~\cite{Dey2017, Munch2016} can only handle continuous topological spaces and cannot be 
used as is for finite metric spaces representing data, such as point clouds and distance matrices.

In this article, we introduce a slight modification of the usual Mapper construction and we give risk bounds for estimating the Reeb space using this
estimator. Our approach applies in particular to the setting where the filter function used to compute Mapper 
is also estimated from data, such as the eigenfunctions of PCA.
Our results are given with respect to the Gromov-Hausdorff distance, computed with specific 
filter-based pseudometrics for Mappers and Reeb spaces defined in~\cite{Dey2017}.
 
We finally provide applications of this setting in statistics and machine learning for different 
kinds of target filters, as well as numerical experiments that demonstrate the relevance of our approach.
\end{abstract}

\section{Introduction}\label{sec:intro}

The \emph{Reeb space} and the \emph{Mapper} are common descriptors of Topological Data Analysis (see for instance \cite{chazal2017introduction}), that can 
summarize and encode the topological features of a given data set using a continuous function, 
often called {\em filter}, defined on it. As such, both objects have been used tremendously in 
many different fields and applications of data science, including, among others, computational 
biology~\cite{Carriere2018, Jeitziner2019, Nicolau2011, Rizvi2017}, computer graphics~\cite{Ge2011, Singh2007}, 
or machine learning~\cite{Bruel-Gabrielsson2018, Naitzat2018}.
Mathematically speaking, the 
Reeb space is a quotient space and the Mapper is a simplicial complex. Both objects are representatives of 
the topology of the input data set, in the sense that any topological feature that is present in these 
objects witnesses the presence of an equivalent one in the input data.  Moreover, the Mapper can be 
thought of as a more tractable approximation of the Reeb space, which, as a quotient space, might be 
difficult to describe and compute exactly. In the simpler case where the filter function is scalar-valued, 
the Mapper and the Reeb space actually become combinatorial graphs, which is why they are mostly used 
for clustering and data visualization. Actually, even when the filter is multivariate, i.e., when its 
domain belongs to $\R^\esd$ with $\esd > 1$, it is common to only compute the skeleton in dimension 1 of 
the Mapper, so as to make it easy to display and interpret.
Even though computation is easier, restricting to scalar-valued functions can still be a dramatic simplification,
since it happens quite often in practice that either multiple filters jointly characterize the data---as is the case, for instance, 
of multiple driver genes explaining a disease or cell differentiation---or that the filter actually takes values
in spaces more complicated than Euclidean space---as is the case where filter functions are stochastic, making
Mappers and Reeb spaces computed with mere realizations of the filter (in Euclidean space) extremely limited.     

In recent works, different notions of stability and convergence of the Mapper to the Reeb space, 
in the case where the filter function is scalar-valued, have been defined and 
studied~\cite{Bauer2014, Brown2019, Carriere2017, Carriere2018a, DeSilva2016}, under various statistical 
assumptions on how data is generated. The case of multivariate and more general filter functions is however 
much more difficult and less understood, since the singular values of the filter function, which turn 
out to be critical quantities to look at in the analysis, cannot be ordered easily, and as a consequence, 
the natural stratification of data (that could be derived for scalar-valued Morse functions for instance) does not extend. 
Few available results, presented in~\cite{Munch2016} and~\cite{Dey2017}, prove nice 
approximation inequalities for continuous spaces, but unfortunately 
do not apply when data is given as a 
finite metric space, such as a point cloud or a distance matrix,
since those finite metric spaces should be thought of as approximations of the underlying
continuous space, and not the space itself.

Moreover, many of previously cited works only consider the case where the values of the 
filter function (either scalar-valued or multivariate) are known exactly on the data points. 
This will not be the case if the filter function is estimated from data, and thus different from the 
filter function used to compute the target Reeb space, as is the case for instance of PCA filters or density filters
which are abundant in Mapper applications. 
This also happens tremendously in statistics and 
machine learning, where the underlying filter is usually a predictor, that has to be estimated with 
standard machine learning methods. 
As explained in this article, another interesting example is when the interesting and underlying filter 
is given by the (scalar-valued) means, or the (multivariate) histograms, of some conditional probability 
distributions associated to each point in the data set, and that what is given at hand are merely 
single realizations of these distributions. Then, the usual way of computing Mappers will clearly not work, 
especially if these conditional probability distributions have large variances, since single realizations 
are not representative at all of the means, or histograms, of the associated conditional probability distributions.  \\

\textbf{Contributions.} The contribution of this article is two-fold: 
\begin{itemize}
    \item We propose risk bounds for the estimation of the 
    Reeb space with a Mapper-based estimator in the general case, that is, for any type of filter
    whose codomain is a complete and locally compact length space~\cite{Burago2001}.
    For this, we use the Gromov-Hausdorff distance computed with filter-based pseudometrics defined 
    on Mappers and Reeb spaces (and originally introduced in~\cite{Dey2017}). Our results are stated in the context where the filter used to compute the Mapper is only an estimation (usually computed from a random sample 
    of data) of the target filter used to compute the Reeb space. 
    \item  We  propose some methodology for using our Mapper-based estimator. We also provide applications and numerical experiments in statistics and machine learning, as well 
    as examples in which the standard Mapper fails at recovering the correct topology of the data, 
    while using our Mapper-based estimator succeeds at doing so.
\end{itemize}

The plan of this article is as follows: in Section~\ref{sec:background}, we recall the basics of Reeb spaces, Mappers, and we introduce the pseudometrics defined on them. Then, we show  risk bounds for  our Mapper-based estimator  in Section~\ref{sec:estim}. Numerical experiments and applications are presented  in Section~\ref{sec:statmeanmapper}. Finally, we conclude and provide 
future investigations in Section~\ref{sec:conc}. 

\section{Background on Reeb spaces and Mappers}\label{sec:background}

In this section, we recall the definitions of the Reeb spaces and Mappers (Section~\ref{subsec:ReebSpaceMapper}), 
and we introduce the Gromov-Hausdorff distance and the filter-based pseudometrics that we use to 
compare them (Section~\ref{subsec:metrics}).

\subsection{Reeb spaces and Mappers}\label{subsec:ReebSpaceMapper}

Reeb spaces and Mappers are mathematical constructions that enable to simplify and visualize the 
various topological structures that are present in topological spaces, through the lens of a 
continuous function, often called \emph{filter}.\\ 

\textbf{Reeb space.} Given a topological space $\X$ and a continuous function 
$f:\X\rightarrow \Z$, where $(\Z,d_{\Z})$ is a metric space,
the \emph{Reeb space} of $\X$ is an approximation of $\X$ that preserves 
its connectivity structures. When $f:\X\rightarrow\R$ is scalar-valued, it is usually called 
the \emph{Reeb graph}~\cite{Reeb1946}.

\begin{definition}\label{def:reeb}
Let $\X$ be a topological space and $f:\X\rightarrow\Z$ be a continuous function 
defined on it. The \emph{Reeb space} of $\X$ is the quotient space
\begin{equation*}\label{eq:ReebSpace}
    \reeb_f(\X) = \X/\sim_f,
\end{equation*}
where, for all $x,x'\in\X$, one has $x\sim_f x'$ if and only if $f(x)=f(x')$ and $x,x'$ belong 
to the same connected component of $f^{-1}(f(x))=f^{-1}(f(y))$. 
\end{definition}

Moreover, the Reeb space comes with a projection $\pi:\X\rightarrow \reeb_f(\X)$ defined with
$\pi(x)=[x]_{\sim_f}$, where $[x]_{\sim_f}$ denotes the equivalence class of $x$ w.r.t. the 
relation $\sim_f$. Since $f$ is continuous, so is $\pi$. \\

\textbf{Approximation with Mapper.} However, the Reeb space is not well-defined when data is 
given as a finite metric space, i.e., a point cloud or a distance matrix, in which case all preimages 
used to compute the Reeb space are either empty or singletons. To handle this issue, the \emph{Mapper} 
was introduced in~\cite{Singh2007} as a tractable approximation of the Reeb space. We first provide its 
definition for continuous spaces. 

\begin{definition}\label{def:mapper}
Let $\X$ be a topological space and $f:\X\rightarrow\Z$ be a continuous function defined on it. 
Moreover, let $\U$ be a cover of $\im(f)$, that is, a family of subsets $\{U_\alpha\}_{\alpha\in A}$ 
of $\Z$ such that $\im(f)\subseteq \bigcup_{\alpha\in A}U_\alpha$.
Let $\V$ be the cover of $\X$ defined as 
$\V=\{V\subseteq\X\,:\,\exists\alpha\in A\text{ s.t. }V\text{ is a connected component of }f^{-1}(U_\alpha)\}$.
The \emph{Mapper} of $\X,f,\U$ is then defined as
\begin{equation*}\label{eq:Mapper}
    \mapper_{f,\U}(\X)=\mathcal N(\V),
\end{equation*}
where $\mathcal N$ denotes the \emph{nerve} of a cover.
\end{definition}

\textbf{Parameters and extension to point cloud.} When data is given as a finite metric space, 
the connected components are usually identified with clustering, and the nerve is computed by assessing 
a non-empty intersection between several cover elements as soon as there exists at least one point that 
is shared by all these elements. In the remaining of this article, we use graph clustering. More precisely,  
we assume that we have a graph $G$ built on top of our finite metric space, and for each element $U$ of 
the cover $\U$, we use the connected components of the subgraph $G(U)$ to compute the Mapper, where $G(U)$ is defined as
\begin{equation}\label{eq:subgraph}
    G(U)=(V_U,E_U),
\end{equation}
where the vertex set $V_U$ is $\{v\in V(G)\,:\,f(v)\in U\}$ and the edge 
set $E_U$ is $\{(u,v)\in E(G)\,:\,u\in V_U,v\in V_U\}$). When $G$ is set to be the $\delta$-neighborhood 
graph $G_\delta$, this amounts to perform single-linkage clustering~\cite{Murtagh2012} with 
parameter $\delta$, and we let 
\begin{equation}\label{eq:mappergraph}
    \mapper_{f,\U,G_\delta}
\end{equation} 
denote the corresponding Mapper for finite metric spaces. 

Moreover, when $\Z=\R^\esd$, it is very usual to define a cover $\U$ with hypercubes by covering every single 
dimension of $\R^\esd$ with intervals of length $r>0$ and overlap percentage $g\in[0,1]$, and then by 
taking the Euclidean products of these intervals. Note that $r$ and $g$ are often called the {\em resolution} 
and the {\em gain} of the cover respectively. We let $\U(r,g)$ denote this particular type of cover. Note 
however that this strategy becomes quickly very expensive, and thus prohibitive, when the dimension $\esd$ is 
large. Actually, even for moderate values, e.g., $\esd=10$, the computation can become very costly if the 
resolution is too small or the gain is too large. Moreover, from a statistical perspective, such a naive strategy 
requires a number of observations which increases exponentially with the dimension, due to the curse of dimensionality. 
It is thus essential to propose greedy methods to define efficient covers in such situations. In Section~\ref{sec:stochmapper}, we provide alternative 
and computationally feasible strategies to cover the filter domain using thickenings of partitions. \\

It has been shown in recent works~\cite{Brown2019, Carriere2017, Carriere2018a, Munch2016} that the Mapper 
actually approximates the Reeb space under various assumptions and metrics when the filter is scalar-valued. 
In the next section, we introduce the filter-based pseudometrics that we will use for comparing Mappers and 
Reeb spaces with the Gromov-Hausdorff distance.

\subsection{The filter-based pseudometric}
\label{subsec:metrics}

The filter-based pseudometric, introduced in~\cite{Bauer2014, Dey2017}, basically measures the diameter of
continuous paths between any two points {\em relative to the filter values}.

\begin{definition}\label{def:pseudomet}
Let $\X$ be a topological space and $f:\X\rightarrow\Z$ be a continuous function defined on it. 
The {\em filter-based pseudometric} $d_f:\X\times\X\rightarrow\R$ is defined as 
$$d_f(x,x')={\rm inf}_{\gamma\in\Gamma(x,x')}\ {\rm max}_{t,t'\in [0,1]}\ d_{\Z}(f\circ\gamma(t), f\circ\gamma(t'))={\rm inf}_{\gamma\in\Gamma(x,x')}\ {\rm diam}_{\Z}(f\circ\gamma),$$
where $\Gamma(x,x')$ denotes the set of all continuous paths $\gamma:[0,1]\rightarrow \X$ such that $\gamma(0)=x$ 
and $\gamma(1)=x'$, and ${\rm diam}_{\Z}$ denotes the {\em diameter} of a subset of $\Z$.

Similarly, the Reeb space $\reeb_f(\X)$ can also be equipped with a pseudometric $\tilde d_f$ using the 
projection $\pi$. For any two equivalence classes $r,r'\in\reeb_f(\X)$,
$$\tilde d_f(r,r')= d_f(x,x') \text{ for arbitrary }x\in\pi^{-1}(r) \text{ and }x'\in\pi^{-1}(r').$$
\end{definition}

Finally, we also define a pseudometric between the nodes of the Mapper $\mapper_{f,\U}(\X)$. Recall that
the nodes of the Mapper are the vertices of the cover $\mathcal N(\mathcal V)$ (see Definition~\ref{def:mapper}),
and hence each node $v$ corresponds to a connected component of $f^{-1}(U)$ for some $U\in\U$. Thus, we can associate
an arbitrary (but distinct) element $z_v\in U$ for each $v$. Let us introduce  the function $f_\U : V(\mapper_{f,\U}(\X)) \rightarrow \Z $ with $f_\U(v) = z_v$.

\begin{definition}\label{def:pseudomet-U}
The filter-based pseudometric is then defined with
$$\tilde d_{f,\U}(v,v')={\rm inf}_{\gamma\in\Gamma(v,v')}\ {\rm max}_{p,q\in\gamma}\ d_{\Z}(f_\U(p),f_\U(q)),$$
where $\Gamma(v,v')$ denotes the set of all paths between $v$ and $v'$ in $\mapper_{f,\U}(\X)$, that is, $\Gamma(v,v')$
is of the form $$\Gamma(v,v')=\{p_1,\dots,p_n\,:\,n\in\N,p_1=v, p_n=v',(p_i,p_{i+1})
\text{ is an edge of }\mapper_{f,\U}(\X), \forall 1\leq i\leq n-1\}.$$
\end{definition}

According to the following proposition, the spaces $\X$ and $\reeb_f(\X)$ (equipped with these pseudometrics) are actually the same when compared with the Gromov-Hausdorff distance~\cite{Burago2001}.

\begin{proposition}\label{prop:spacereeb}
Let $\X$ be a topological space and $f:\X\rightarrow\Z$ be a continuous function defined on it. Then
$$d_{\rm GH}((\X,d_f), (\reeb_f(\X),\tilde d_f))=0.$$
\end{proposition}

\begin{proof}
Since $\pi$ is surjective, let $\mathcal C$ be the correspondence between $\X$ and $\reeb_f(\X)$ defined with $\pi$, that is, $\mathcal C=\{
(x,\pi(x))\,:\,x\in \X)\}$.
Let $x,x'\in\X$. Then the metric distortion of $x,x'$ induced by $\mathcal C$ is $\mathcal D(x,x'):=|d_f(x,x')-\tilde d_f(\pi(x),\pi(x'))|=0$ by definition of $\tilde d_f$, 
and then 
$d_{\rm GH}((\X,d_f), (\reeb_f(\X),\tilde d_f))\leq {\rm sup}_{x,x'\in\X}\ \mathcal D(x,x')=0$.
\end{proof}

Let $\res(\U,f)={\rm max}_{\alpha\in A}\ {\rm sup}_{u,v\in U_\alpha \cap \im(f)} d_\Z(u,v)={\rm max}_{\alpha\in A}\ {\rm diam}_{\Z}(U_\alpha \cap \im(f))$ the resolution of $\U$ in $\Z$ with respect to the filter $f$.
It turns out that Mappers and Reeb spaces equipped with their respective pseudometrics are actually close
for covers with small diameters.

\begin{theorem}[\cite{Dey2017}]\label{thm:convdgh}
Let $\X$ be a topological space and $f:\X\rightarrow\Z$ be a continuous function defined on it. Then
$$d_{\rm GH}((\mapper_{f,\U}(\X), \tilde d_{f,\U}),(\reeb_f(\X), \tilde d_f))\leq 5 \cdot \res(\U,f) .$$
\end{theorem}

\section{Reeb space inference }
\label{sec:estim}

In this section we propose a new Mapper-based estimator for Reeb spaces in the general framework where the domain $(\X, d_{\X})$ and the codomain $(\Z, d_{\Z})$ of the filter function $f$ are  
complete and locally compact length 
spaces~\cite{Burago2001}, 
which in particular means that
shortest paths exist for every pair of points in $\X$ and $\Z$ (as per Theorem~2.5.23 in~\cite{Burago2001}).
We recall that a shortest path between $z$ and $z'$ in $\Z$ is a continuous path $\gamma^*:[0,1]\rightarrow\Z$ with $\gamma^*(0)=z$, $\gamma^*(1)=z'$, such that, for any other continuous path
$\gamma:[0,1]\rightarrow\Z$ with $\gamma(0)=z$, $\gamma(1)=z'$, one has $|\gamma^*|\leq |\gamma|$ and $|\gamma^*|=d_\Z(z,z')$, where $|\cdot|$ denotes the length associated to $\Z$.



The main idea behind our estimator is to first compute a refinement of the input point cloud in order to remove its pathological elements (w.r.t. the cover used for computing the estimator). These elements are the so-called {\it element-crossing edges}, defined in Section~\ref{sub:calib_risk}. Then, our estimator is defined as the standard Mapper estimator for this refined point cloud. 
We first introduce a raw version of the estimator without calibrating the parameters. 
This calibration is then detailed further and allows to provide a risk bound for our corresponding estimator. 

\subsection{A Mapper-based estimator}
\label{sec:stochmapper}


In this section, we introduce our Mapper based estimator in a deterministic setting.
Assume that two point clouds $\hX_n$ and $\hZ_n$ are given: $\hX_n = (x_1, \dots, x_n)$ and $\hZ_n = (z_1, \dots,z_n)$ such that for any $i$, one has $(x_i,z_i) \in \X \times \Z$ and $z_i = \hat f (z_i)$. The function $\hat f :  \hX_n \rightarrow \Z$ is an approximation of a ``true" and unknown filter function $f:\mathcal X\rightarrow\Z$. 
In some settings, the true exact filter $f$ is known and then we simply have $\hat f = f\vert_{ \hX_n}$.


\paragraph{Point cloud and embedded graph.} We let $G_{\delta}$ be the (metric) neighborhood graph 
built on top of $\hX_n$ with parameter $\delta$, that is, any pair $\{x_i,x_j\}\subseteq\hX_n$ creates an edge 
in $G_\delta$, with length $d_\X(x_i,x_j)$, and parameterized with a shortest path between $x_i$ and $x_j$, if and only if $d_\X(x_i,x_j)\leq\delta$. We then define the corresponding {\em embedded graphs} as $G^{\Z}_\delta$ (resp. $\hat G^{\Z}_\delta$) in $\Z$, 
with vertices $\{f(x_i)\,:\,x_i\in \hX_n\}$ (resp. $\{\hat f(x_i)\,:\,x_i\in \hX_n\}$) and 
whose edges are geometric realizations of edges of $G_\delta$, that is, shortest paths in $\Z$\footnote{Our construction actually works for arbitrary paths. However, some quantities that are necessary for computing the estimator (such as $\ell$) might be easier to compute when working with shortest paths, so we stick to those paths in this article. }.  When $\Z$ is a normed vector space (such as a Banach space), this corresponds to linear interpolations in $\Z$. 
We finally extend $f\vert_{ \hX_n}$ (resp. $\hf$) to a function 
$\fpl:G_\delta\rightarrow G^{\Z}_\delta$ (resp. $\hfpl:G_\delta\rightarrow \hat G^{\Z}_\delta$), which maps the interiors
of the edges of $G_\delta$ to the corresponding interiors of the shortest paths in $\Z$.
Note that $f|_{\hX_n}$ and $\fpl$ (resp. $\hf$ and $\hfpl$) coincide on $\hX_n$, so we will only use $\fpl$ (resp. $\hfpl$) when applied to (interiors of) edges of $G_\delta$.   
 
\paragraph{Graph refinement.} Our Mapper-based estimator is defined as the standard Mapper computed on 
a refinement of the graph $G_\delta$. For $s \in \mathbb N^*$, we subdivide each edge of $G_\delta$ with $s$ points. 
Let $G_{\delta,s}$ be the resulting graph (on which $\fpl$ and $\hfpl$ are still well-defined), $\hX_{n,s}$ be the refined point cloud, and $G^{\Z}_{\delta, s}$, $\hat G^{\Z}_{\delta, s}$ be 
the refined embedded graphs. 

 

\paragraph{Cover.} Let $\U$ be a finite cover of $\im(\hfpl)$, which can be data dependent. For now, we assume that this cover is given.  We will discuss the construction of $\U$ further in this article.


\paragraph{Estimator.} We are now in position to define our Mapper based estimator. It is defined with~(\ref{eq:mappergraph}) as:
\begin{equation} \label{def:estimator}
   \mapper_n=\mapper_{\hfpl,\U, G_{\delta,s}}(\hX_{n,s}),
\end{equation}
and it can  be equipped with the pseudometric $\tilde d_{\hfpl,\U}$ of Definition~\ref{def:pseudomet-U}. See Figure~\ref{fig:estimator} for an illustration of the construction of $\mapper_n$. For defining the estimator, we need to choose the scale parameters $s$ and $\delta$, and we discuss this question further in the next section.

\begin{figure}[h]
    \centering
    \includegraphics{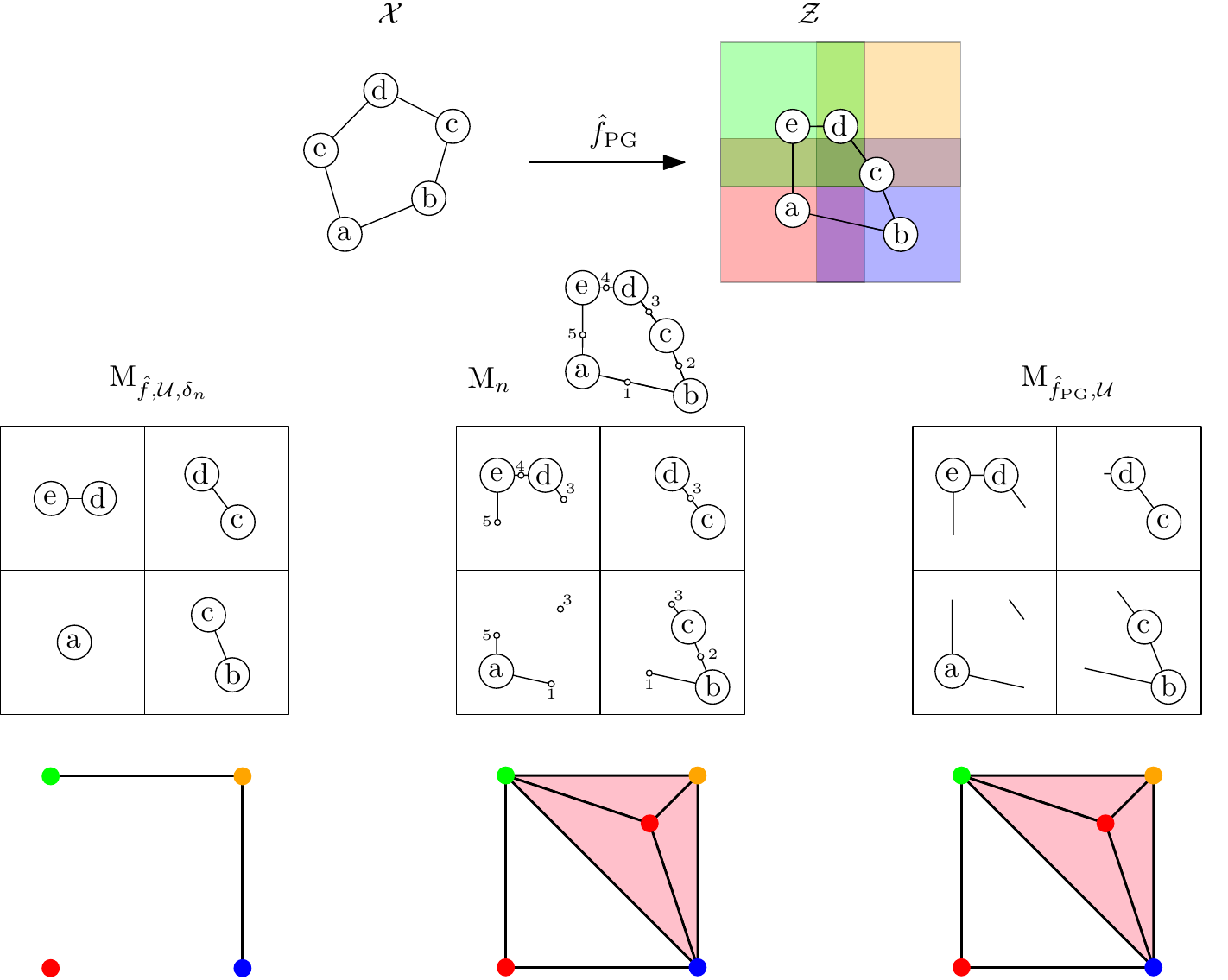}
    \caption{Example of our estimator $\mapper_n$ on a dataset of 5 points. Upper left: Dataset $\hX_n$ with five points $a,b,c,d,e$. The edges between the points are computed with a neighborhood graph with parameter $\delta_n$. Upper right: Cover of ${\rm im}(\hfpl)$ with four squares. The edges between the points in $\Z$ are shortest paths in $\Z$. Lower left: Preimages of the four squares for the standard Mapper on the point cloud and corresponding simplicial complex computed with hierarchical clustering with parameter $\delta_n$. Lower right: Preimages of the four squares for the standard Mapper on the metric neighborhood graph. Lower middle: Our estimator is computed by refining the neighborhood graph (with five extra nodes $1,2,3,4,5$) and by using this new graph to compute the connected components and the intersections between them.}
    \label{fig:estimator}
\end{figure}


\subsection{Risk bound and parameter calibration}
\label{sub:calib_risk}

We now give our main result in the {\em Stochastic Filter setting}, i.e., when the data $X_1,\dots, X_n$ are sampled i.i.d. from a distribution $P$ and when the function $\hat f :  \hX_n \rightarrow \Z$ is allowed to be data dependent. In this setting the $Z_i$'s are thus also i.i.d. random variables. 
\begin{itemize}
    \item {\bf (H1) Support Assumption.} The support $\X$ of $P$ is a compact 
    submanifold $\X\subseteq \R^D$ with 
    positive 
    reach 
    $\reach(\X)$ 
    (see~\cite{Boissonnat2019b} for definitions).
 \end{itemize}    
The neighborhood graph $G_\delta$ is built with Euclidean norm $\|\cdot \|$ in $\R^D$. Let $D_\X < \infty$ denote the diameter of $\X$ (in the Euclidean distance). Hereafter, we call 
$\reach(\X)$ and $D_\X$ the {\it geometric parameters of  $\X$}.

\begin{itemize}
    \item {\bf (H2) Measure Assumption.} The probability measure $P$ is $(a,b)$-standard, 
    i.e., $P(B(x,r))\geq{\rm min}\{1, ar^b\}$, for all $x\in\X$ and $r>0$, 
    where $B(x,r)=\{y\in\R^D\,:\,\|y-x\|\leq r\}$.
 \end{itemize}

By assumption, the filter of a Reeb graph is a continuous function. The filter function $f$ is thus uniformly continuous on the compact set $\X$ and it admits a minimal modulus of continuity $\omega_f$. The function $\omega_f$ is a non-decreasing function such that for any $u \in \R^+$,
$$ \omega_f(u) = \sup \left\{ d_{\Z}(f(x),f(x')) \, : \, 
  (x,x')\in \X^2 \text{ and } \| x - x' \| \leq u  \right\}.
$$ 
The domain $\X$ being compact, $\omega_f$  satisfies (see for instance Section~6 in \cite{DeVore93})
\begin{enumerate}
\item $ \omega_f(\delta)  \rightarrow \omega(0) = 0$  as $ \delta \rightarrow 0$;
\item $\omega_f$ is non negative and non-decreasing on $\R^+$;
\item $\omega_f$ is subadditive : $\omega_f(\delta_1 + \delta_2) \leq \omega_f(\delta_1) +\omega_f(\delta_2) $ for any $\delta_1$, $\delta_2 >0$;
\item $\omega_f$  is continuous on $\R^+$.
\end{enumerate}
In this article, we say that a function $\omega$ defined on $\R^+$ is \emph{a modulus of continuity} if it satisfies the four properties above 
and we say that  it is \emph{a modulus of continuity for $f$} if, in addition, we have
\begin{equation*}
 |f(x) -  f(x')| \leq \omega(\|x - x'\|),
\end{equation*}
for any $x,x' \in \X$. 
\begin{itemize}
    \item {\bf (H3) Filter Regularity Assumption.} The true filter $f:\X\rightarrow\mathcal \Z$ is a continuous function on $\X$ which admits a modulus of continuity $\omega$ such that $x \in \R^+ \mapsto \frac{\omega(x)}x$ is a non-increasing function on $\R^+$. 
 \end{itemize}

Finally,  we will assume the following assumption on the cover:
\begin{itemize} 
    \item {\bf (H4) Cover Assumption.} The cover $\U$ is assumed to cover $\im(\hfpl)$. 
 \end{itemize}  
 
\medskip 

For calibrating  the estimator parameters, we need to introduce the notion of \emph{element-crossing edges}. Such edges are pathological in the sense that they  may prevent to recover the correct topology of the underlying Reeb space. Given a simplex $\sigma=\{U_{\alpha_1},\dots,U_{\alpha_p}\}$ in the nerve $\mathcal N(\U)$, we let $U_\sigma=\cap_{i=1}^p U_{\alpha_i}$. 

\begin{definition} Let $(X_i,X_j) \in\hX_n^2$ such that the edge $e = (X_i,X_j)$ belongs to $G_\delta$. 
Let $\hfpl(e)$ be the corresponding edge in $\hat G^{\Z}_\delta$.
We say that $e$ is an \emph{element-crossing edge  with respect to the cover $\U$} if 
there exists $\sigma\in\mathcal N(\U)$ such 
that $\hfpl(e)\cap U_\sigma\neq\emptyset$, $\hf(X_i)\not\in U_\sigma$ and $\hf(X_j)\not\in U_\sigma$. 
\end{definition}
In other words, $e$ is an element-crossing edge  with respect to the cover $\U$ if the shortest path $\hfpl(e)$ goes through $U_\sigma$, even though its endpoints $\hf(X_i)$ 
and $\hf(X_j)$ are outside $U_\sigma$. In this case we say that $U_\sigma$ is \emph{crossed} by $e$. Note that element-crossing edges are generalizations of \emph{interval} and \emph{intersection-crossing edges}, 
as defined in~\cite{Carriere2017}.
We then define:
    $$\ell(\hX_n,\hf,\U)={\rm inf}\left\{|\tilde e | \, , \, \text{where $\tilde e$ is a c.c. of $\hfpl(e)\cap \U_\sigma$, $e$ is element-crossing and $U_\sigma$ is 
    crossed by $e$} \right\},$$
     where $|\cdot|$ denotes the length in $\Z$ and c.c. is a shorthand for connected component. In other words, $\ell(\hX_n,\hf,\U)$ is the length of the smallest connected path in the intersection between an edge of $\hat G^{\Z}_\delta$ and a cover element or intersection, such that the edge endpoints do not belong
    to this cover element or intersection.
    
\medskip

For calibrating the parameters, we also need to introduce the modulus of continuity of $\hfpl$: 
$$ \hat \omega_{\rm PG}(h) = \sup \left\{d_\Z (\hfpl(x),\hfpl(x')) \,:\,  
\|x-x'\| \leq h\textrm{ and }x,x'\textrm{ belong to the same edge of }G_\delta  
\right\},
$$
where $|\cdot|$ denotes the edge length in $G_\delta$.

\medskip

We are now in position to define the calibrations for $\delta$ and $s$. We follow a similar strategy as in \cite{Carriere2018a}. Let $d_{\rm H}^E$  denotes the Hausdorff distance~\cite{Burago2001} 
computed with Euclidean distances.
 \begin{itemize}
    \item  {\bf Choice for $\delta$.} For some arbitrary $\beta>0$, let $s(n)=n/({\rm log}(n))^{1+\beta}$. We take
\begin{equation}
    \label{eq:deltan}
\delta = \delta_n=d^E_{\rm H}( \tilde \hX_{s(n)},\hX_n) \: ,
\end{equation}
where $\tilde \hX_{s(n)}$ is a random subsample of size $s(n)$ drawn uniformly from $\hX_n$ with replacement.
    \item  {\bf Choice for $s$.} Let $\ell  =  \ell(\hX_n,\hf,\U)$, we take
\begin{equation}
    \label{eq:sn}
    s \geq s_n := \left\lfloor \frac{\delta_n}{\hat{\omega}_{\rm PG}^{-1}(\ell/2)}\right\rfloor \text{ if } \ell/2\in{\rm  im}(\hat\omega_{\rm PG})
    \end{equation}
   and $s_n=0$ (that is, we do not refine $G_\delta$) otherwise. By convention, we also  let $s_n=+\infty$ if $\ell =0$, which happens with null probability.
\end{itemize}

Under the previous assumptions and with the definitions of $s_n$ and $\delta_n$ given above, we can provide the following risk bound of our Mapper based estimator:
\begin{theorem}\label{th:stochmapper}
Under assumptions {\rm (H1)}, {\rm (H2)}, {\rm (H3)} and  {\rm (H4)},  the following inequality is true: 
\begin{equation*} \label{eq:general_risk_bound}
    \E\left[d_{\rm GH}((\mapper_{n},\tilde d_{\hfpl,\U}), 
    (\reeb_f(\X),\tilde d_f))\right] \leq 5 \E\left[\res(\U,\hfpl)\right] + 
    C \omega\left(C'\frac{{\rm log}(n)^{(2+\beta)/b}}{n^{1/b}} \right) +
    2 \E\left[\|\fpl-\hfpl\|_\infty\right],
\end{equation*}
where the constants $C,C'$ only depends on $a$, $b$ and the geometric parameters of $\X$, and
where the third term is defined with 
$\|\fpl-\hfpl\|_\infty := {\rm sup}_{x\in\Gn}d_\Z(\fpl(x), \hfpl(x))$.
\end{theorem}

The proof is given Appendix~\ref{app:bound}. Note that this result is very general and can handle random or data dependent covers for instance. 
We discuss cover choices and corresponding upper bounds on their resolutions in Section~\ref{sec:covcont}.
Moreover, even though the third term might be difficult to control for general length spaces, when $\Z$ is a Hilbert space  it actually reduces to $$\E\left[\|\fpl-\hfpl\|_\infty\right]=\E\left[\|(f-\hf)|_{\hX_n}\|_\infty\right],$$
since shortest paths are straight lines in such spaces. This is the case for instance when $\Z$ is a reproducing kernel Hilbert space (RKHS), which we study in more details further in Section~\ref{sec:kpca}.
 

\paragraph{Parameter calibrations.} Theorem~\ref{th:stochmapper} relies on the calibration of the parameters of our Mapper-based estimator $\mapper_n$. In particular, the choice we make for the graph refinement parameter $s$ requires to: first, upper bound the modulus of continuity $\hat{\omega}_{\rm PG}$ of $\hfpl$, and second, to compute the smallest connected path $\ell(\hX_n,\hf,\U)$. Controlling $\hat{\omega}_{\rm PG}$  is not possible in general, but for standard filters such as KPCA filters (see Section~\ref{sec:kpca}), $\hf$ and $\hfpl$ are Lipschitz functions and hence $\hat{\omega}_{\rm PG}$ can be easily bounded by the corresponding Lipschitz constant.  Next, computing---or at least lower bounding---the quantity $\ell(\hX_n,\hf,\U)$ is difficult for a general cover $\U$. However, it can be done exactly for particular covers, such the ones induced by thickening $K$-means or Voronoi 
partitions in Hilbert spaces (see Section~\ref{sec:covcont}). Indeed, in this case, it is possible  to test whether a given shortest path intersects a cover element 
or intersection by computing the intersection of the line induced by the shortest path---which is possible since shortest paths are segments---and all the mediator lines that 
form the boundary of the cover element. 

In practice, when $\hat{\omega}_{\rm PG}$ and $\ell(\hX_n,\hf,\U)$ are difficult to compute, we adopt a conservative approach by 
considering  for the graph refinement parameter $s$  the largest possible integer that still allows our estimator to be computed with a reasonable amount of time and memory usage, depending on the machine that 
is being used. Finally, it should be noted that small sizes of cover elements or intersections induce small 
$\ell$ and large $s$, and thus potentially longer computation times.

\subsection{Cover control}
\label{sec:covcont}
 
In this section, we study the resolutions of covers induced by Voronoi partitions. In particular, we define those covers in Section~\ref{sec:vorpart}, and provide upper bounds on their resolutions in Section~\ref{subs:reso}. This allows to formulate more explicit upper bounds for the first term in Theorem~\ref{th:stochmapper}. 

\subsubsection{Defining covers}\label{sec:vorpart}

Covers with hypercubes  is the most common cover used with Mappers when $\mathcal Z = \R^\esd$ when $\esd$ is not too large. When the filter domain $\Z$ is a general length space, as for instance when $\Z$ is the space of probability distributions of $\R$ (see Section~\ref{sec:single}) or when $\Z$ is a space of combinatorial  graphs (see Section~\ref{sec:metricex}), we need an alternative construction to define covers.
A simple way of generating a cover is by using a partition of this space, and thickening the elements of this partition.
\begin{definition}
Let $(\Z, d_{\Z})$ be a length space, and let $\epsilon > 0$.
\begin{itemize}
    \item For $U \subseteq \Z$   a subset of $\Z$, the {\em $\epsilon$-thickening} of $U$ is defined as 
$U^\epsilon=\{z\in\mathcal Z\,:\,{\rm inf}\{d_{\Z}(z,\tilde z)\,:\,\tilde z \in U\}\leq\epsilon\}$. 
\item Let $\widetilde{\mathcal Z}$ be a subset of $\Z$ and let $\mathcal U=\{U_\alpha\}_{\alpha\in A}$ be a partition of $\widetilde{\mathcal Z}$, i.e., $\widetilde{\mathcal Z} \subseteq \bigcup_{\alpha\in A} U_\alpha$ and $U_\alpha\cap U_\beta=\emptyset$ for all $\alpha\neq\beta\in A$. The {\em $\epsilon$-thickening} of $\mathcal U$ for covering $\widetilde{\mathcal Z}$ is defined as 
$\mathcal U^\epsilon = \{U_\alpha^\epsilon\}_{\alpha \in A}$.
\end{itemize}
\end{definition}

Even when $\Z=\R^\esd$, it might be interesting  to use thickenings of partitions instead of hypercube covers, since the number of hypercubes increases 
exponentially with the dimension $\esd$, with many of them having an empty preimage under $f$, and thus useless. 

Note also that when our Mapper-based estimator is used with an $\epsilon$-thickening cover, our estimator gets more difficult to compute when the thickening parameter $\epsilon$ goes to zero, since it requires refining the initial neighborhood graph with a lot of new vertices.  

\subsubsection{Bounding the resolution for  $\epsilon$-thickening of Voronoi partitions} \label{subs:reso}

Under the same general assumptions as Section~\ref{sub:calib_risk}, we consider the specific case where $\mathcal Z  = \R^\esd$ endowed with the inner product $\langle \cdot,\cdot \rangle $. Partitions of $\R^\esd$ can be computed very efficiently, for instance with Voronoi partitions 
and the $k$-means algorithm. In this section we give an upper bound on the resolution term involved in the upper bound of Theorem~\ref{th:stochmapper}, for a cover computed from a $k$-means algorithm.

\begin{definition}
For $Q$ a measure in $\R^\esd$ and $k \in \mathbb N^*$, a set $t(Q)$ of $k$ points in $\R^\esd$ is said to be $k$-optimal for $Q$ if 
\begin{equation*}
    t(Q) \in \underset{t \in (\R^\esd)^k}{\rm argmin} \int_{\mathcal Z} \min_{i=1, \dots, k} \| z - t_i  \|_2^2 \, d Q(z).
\end{equation*}   
\end{definition}
Let $P_n^{\hat f}$ be the push forward measure of $P_n$ by the stochastic filter function $\hat f$, where $P_n$ is the empirical measure associated to $\hX_n$. Of course, all these quantities are defined conditionally to the sample $\hX_n$ and $\hat f$ (in particular if $\hat f$ depends on other observations $\mathbb Y_n$, as would be the case, e.g., in the context of a regression function filter). Note that $P_n^{\hf}$ is equivalently defined as the empirical measure  corresponding to the observation of the sample $\hZ_n = \hf(\hX_n)$. The $k$-means algorithm on $\hZ_n$ aims at approximating an optimal $k$ points for the empirical measure $P_n^{\hat f}$ from the observation  $\hZ_n$.

Let $\hat t = t(P_n^{\hat f})$ be an optimal $k$-points  for the empirical measure $P_n^{\hat f}$. We denote by $\widehat{\mathcal U}^\epsilon = \{\hat U_j^\epsilon  \}_{j = 1, \dots, k}$ the $\epsilon$-thickening of the Voronoi partition associated to $\hat t$. 
Since $\mathcal Z  = \R^\esd$, we know that $\hfpl$ is a linear interpolation between the  $Z_i$'s. Thus $\widehat{\mathcal U}^\epsilon$ is a cover for $\im(\hfpl)$ and   Assumption (H4) is satisfied. 

We give our result  under the additional assumption that the minimal modulus of continuity $\omega$ is upper bounded by a concave function of the form  $\bar{\omega}(u) =  c u^\gamma$, with $c\geq 0$ and $0<\gamma\leq 1$. This assumption is obviously stronger than Assumption (H3).
\begin{itemize}
\item {\bf (H5) Power function upper bounds  $\omega$.} There exists $\gamma \in (0,1]$ and $c \in \R^+$ such that for any $u \in \R^+$, $\omega(u) \leq  \bar{\omega}(u) =  c u^\gamma$.
\end{itemize}
This technical assumption allows us to provide a  simple  upper bound. Note  that it makes sense because on the compact set $\X$, the minimal modulus of continuity $\omega$ is indeed a concave function.  
The next result gives a control on the resolution of $\widehat{\mathcal U}^\epsilon$ in $\R^\esd$ with respect to the filter function $\hfpl$. 
\begin{theorem}
\label{theo:resolution}
Under assumptions {\rm (H1)}, {\rm (H2)} and {\rm (H5)},  for $\mathcal Z  = \R^\esd$ and for $k\leq \frac n {\esd+2}$,  the resolution of the cover $ \widehat{\mathcal U}^\epsilon $ in $\R^\esd$  with respect to the filter function $\hfpl$ satisfies
\begin{equation} \label{res:theo}
\E\left[\res(\widehat{\mathcal U}^\epsilon,\hfpl)\right] \leq  
C_1    \left[ k^{-\frac{2 \gamma^2}{b^2+ 2 \gamma b}} +  \left( \frac {k \esd }{n} \right) ^{\frac \gamma{2b+4\gamma}}  +   \E \| (f - \hat f)\vert_{\hX_n}  \|_\infty  \right]     +   2 \varepsilon  .
\end{equation}
Consequently, the following risk bound holds for our Mapper based estimator $\mapper_n=\mapper_{\hfpl,\widehat{\mathcal U}^\epsilon, G_{\delta,s}}(\hX_{n,s})$:
\begin{equation} \label{mapper-res:theo}   \E\left[d_{\rm GH}((\mapper_{n},\tilde d_{\hfpl,\widehat{\mathcal U}^\epsilon}), 
    (\reeb_f(\X),\tilde d_f))\right] \leq
    C_2    \left[ k^{-\frac{2 \gamma^2}{b^2+ 2 \gamma b}} +  \left( \frac {k \esd }{n} \right) ^{\frac \gamma{2b+4\gamma}}  +   \E \| (f - \hat f)\vert_{\hX_n}  \|_\infty  \right]     +   10 \varepsilon  .
\end{equation}
Moreover, the constants $C_1$ and $C_2$ depends on $a$, $b$, $c$, $\gamma$,  $\|f\|_{\infty}$   and on the geometric parameters of $\X$.   
\end{theorem}
The proof of Theorem~\ref{theo:resolution} is given in Section \ref{sub:proof_resolution}, in which several ideas from \cite{brecheteau2020k} are reused and adapted. Note that we could also provide a  deviation bound on the resolution by applying the so-called Bounded Inequality in a standard way (see for instance Theorem 6.2 in \cite{boucheron2005theory}).

\paragraph{Rate of convergence.} Assuming that $\gamma$  and $b$ are known, we  can choose $k$ to balance the first two terms in the bracket of the right hand side of Inequality~\eqref{mapper-res:theo}. By taking $k $ of the order of $ \left( \frac {n}{\esd} \right) ^{\frac {b^2+2\gamma b}{(b+2\gamma )(4 \gamma  +b)}} $, we obtain that the first two  terms in the bracket are of the order of $ \varepsilon_n := \left( \frac {\esd}{n} \right) ^{\zeta}  $ with  $\zeta  := \frac {2\gamma^2}{(b+2\gamma )(4 \gamma  +b)} < \frac 14$.  If the convergence of $\hat f $ to $f$ is faster than  $\varepsilon_n$, and taking a resolution $\varepsilon$ of the order of $\varepsilon_n$, we finally obtain that the expected risk of our Mapper based estimator is of the order of $\varepsilon_n$. We conjecture that this rate  of convergence is not optimal, however it can be used to show the consistency of our Mapper-based estimator. 
 
\subsection{Application to KPCA filters} 
\label{sec:kpca}

In this section, we study the upper bounds of Theorem~\ref{th:stochmapper} in the particular case where $\Z$ is a reproducing kernel Hilbert space (RKHS).
Let $\Z$ be a RKHS  associated to a continuous kernel function $K$ defined on $\mathcal X \times \mathcal X$. The set $\mathcal X$ being compact and $K$ being continuous, $\mathcal Z$ is then a separable  RKHS. Moreover, the feature map $x \mapsto K(x,\cdot)  $ is a continuous function from $\mathcal X$ to $\mathcal Z$ since:
\begin{equation}
\label{continuityK}
\| K(x,\cdot) - K(x',\cdot)\|_{\Z}
^2 = K(x,x)+ K(x',x') - K(x,x')- K(x',x).
\end{equation} 
Moreover, the random variable $Z = K(X,\cdot)$ is bounded in $\Z$ since $ \| K(X,\cdot) \|_{\Z}
^2 = K(X,X)$ which is almost surely bounded on the compact space $\mathcal X$. In particular, $\E \left[\|  K(X,\cdot) \|_{\Z}^2 \right] < \infty $. In this setting, the covariance operator of the distribution of $Z$ is well defined (see for instance Section 2 and 4.1 in \cite{blanchard2007statistical}). 
To simplify, we will assume that the distribution of $Z$ is centered. 

\paragraph{Covariance operator.} Let $\Gamma = \E (Z \otimes Z^*)$ be the covariance operator and let $\Pi_\esd$ be the orthogonal projection operator on  the set of the first $\esd$  eigenvectors of $\Gamma$. The operator $\Gamma$ can be approximated by its empirical version:
$$ \Gamma_n = \frac 1n   \sum_{i=1}^n Z_i \otimes Z_i^*   $$ 
where $Z_i = K(X_i,\cdot)$.
Let $\hat \Pi_{n,\esd}$ be the orthogonal projection operator on  the set of the first $\esd$  eigenvectors of $\Gamma_n$.
In this section, we consider as filter functions the composition of the feature map with one of the two projection operators:
$$ f_\esd : x \in \mathcal X \mapsto \Pi_\esd\left(K(x,\cdot)\right) $$ 
and
$$ \hat f_{n,\esd} : x \in \mathcal X \mapsto \hat \Pi_{n,\esd}\left(K(x,\cdot)\right) .$$

\paragraph{Modulus of continuity.} Let $\omega_K$ be the modulus of continuity of $K$: for any $(x_1,x_2,x_1',x_2') \in \mathcal X^4$,
$$ \left| K(x_1,x_2) - K(x_1',x_2') \right| \leq \omega_{\tiny K}\left( \sqrt{\|x_1-x_1'\|  ^2 + \|x_2-x_2' \|  ^2} \right), $$
where  $\|\cdot \|$ is the euclidean norm of $\R^D$. Let $(x,x') \in \mathcal X^2$, then
\begin{eqnarray*}
\left\| f_\esd(x) - f_\esd(x') \right\|_{\Z}
&  =  &  \left\| \Pi_\esd\left( K(x,\cdot)\right) - \Pi_\esd\left(K(x',\cdot)\right) \right\|_{\Z} \\
&  \leq   &  \left\|   K(x,\cdot) -K(x',\cdot)   \right\|_{\Z} \\
&  \leq   &  \sqrt{2 \omega_K(\|x-x'\|) }
\end{eqnarray*}
where the last inequality comes from \eqref{continuityK}. This shows that $\sqrt{2 \omega_K} $ is a modulus of continuity for $f_\esd$.

\paragraph{Upper bound.} The statistical analysis of PCA in Hilbert spaces has been the subject of several works, see for instance \cite{reiss2020nonasymptotic, blanchard2007statistical, biau2012pca, shawe2005eigenspectrum}. Here, we need a control for the sup norm between the filter and its empirical version.  According to Theorem 2.1 in \cite{biau2012pca}: 
\begin{eqnarray*}
\E  \left[ \sup_{ z \in \Z \, , \, \| z\|_{\Z} \leq 1} \| \Pi_\esd(z) - \hat \Pi_{n,\esd}(z)  \|_{\Z} \right] \leq \frac{C}{\sqrt n} 
\end{eqnarray*}
where the constant $C$ only depends on $\esd$. Since $\mathcal X$ is compact and $x \mapsto K(x,\cdot)$ is continuous, it follows that:
\begin{eqnarray*}
\E   \left[ \sup_{ x \in \X } \| f_\esd(x) - \hat f_{n,\esd}(x)  \|_{\Z} \right] \leq \frac{C'}{\sqrt n} 
\end{eqnarray*}
where $C'$ depends on $D_{\mathcal X}$ and $\esd$. Under assumptions (H1), (H2) and (H5), if we perform a $k$-means algorithm in the space of  the $\esd$ first components of the KPCA to derive a cover as explained in Section~\ref{subs:reso}, we can then apply Theorem~\ref{theo:resolution} to our corresponding Mapper-based estimator.  The convergence of the estimated filter in $O(1/\sqrt n)$ is fast enough so that it does not slow down the convergence of our Mapper-based estimator.
For $k$ and $\varepsilon$ chosen as in the discussion following Theorem~\ref{theo:resolution},  we finally obtain that the risk of our Mapper based estimator can be upper bounded by a term of the order of $\left( \frac {\esd}{n} \right) ^{\frac {2\gamma^2}{(b+2\gamma )(4 \gamma  +b)}}$.
 

\section{Applications of Mapper in the Stochastic Filter setting}  
\label{sec:statmeanmapper}

In this section, we focus on examples and applications of the Stochastic Filter setting (see Section~\ref{sec:estim}), in which the filter $\hf$ used to compute the Mapper is assumed to be an estimation (computed from the data sample) of the true target filter $f$ used to compute the Reeb space. 
We first provide in Section~\ref{sec:MLmapp} various examples of stochastic filters in statistics and machine learning. Indeed, standard methods provide estimated regression functions and classification probability estimates which are interesting to study with Mapper. Then, we turn the focus to the length space of probability distributions in Section~\ref{sec:single}, 
and we finally provide an illustration for the length space of combinatorial graphs with the graph edit distance in Section~\ref{sec:metricex}. Throughout this section, the Mappers that are computed and discussed always refer to our Mapper-based estimator.

\subsection{Stochastic Filter in Statistical Machine Learning}
\label{sec:MLmapp}

In this section, we discuss the various potential applications of Mappers in statistical machine learning, in which the filter is often used for inference and prediction, and we provide associated numerical experiments and illustrations. We also refer the interested reader to~\cite{Hastie2003} for more details on the statistical and machine learning methods used in this section. \\

{\bf Stochastic real-valued filters.} We first consider a few applications in which the estimated and true target filters are real-valued functions, i.e., $\Z = \mathbb R$. In this setting, one can apply either the risk bound given in Theorem~\ref{th:stochmapper}  or the results from~\cite{Carriere2018a} to quantify the approximation and convergence of Mapper.  

\begin{itemize}
    \item {\bf Inference.} When the target filter function only depends on the measure $P$ itself, we can define estimators of this filter using the point cloud $\hX_n$ alone. For instance, a dimension reduction filter (e.g. PCA), the eccentricity filter or the density estimator filter are all estimators of underlying filters defined from $P$. See for instance \cite{Carriere2018a} for examples.
    
     \item {\bf Regression.} We now assume that we observe a random variable $Y_i$ at each point $X_i$: 
     \begin{equation}
     \label{eq:reg}
         Y_i = f(X_i) + \varepsilon_i, \quad i =1, \dots, n
     \end{equation}
     where the true filter is $f(x) = \E(Y|X=x)$, i.e., the regression function on $\X$ and $\varepsilon_i = Y_i  - f(X_i)$. Then, the Mapper of $\hX_n$ can be computed with any estimator $\hat f$ of $f$ (from the statistical regression literature) in order to infer the Reeb space $\reeb_f(\X)$. 
     \item {\bf Binary classification.}  We now assume that we observe a binary variable $Y_i\in \{-1,1\}$ at each point $X_i$ of the sample. Let $f(x) = P(Y=1|X=x)$ be the probability of class $1$ for any $x\in \X$. In this setting, inferring the target Reeb space $\reeb_f(\X)$ with a Mapper computed on $\hX_n$ for some estimator $\hat f$ of the class probability distribution (given by any machine learning classifier) would provide insights about how data is topologically stratified w.r.t. the confidence given by the classifier.
\end{itemize}


{\bf Extension to stochastic multivariate filters.} 
For many problems in statistical machine learning, the quantity of interest is actually a  multivariate quantity. In this setting, using Theorem~\ref{th:stochmapper}
allows to statistically control the quality of Mapper, which, to our knowledge, is new in the Mapper literature. 

\begin{itemize}
    \item {\bf Dimension reduction.} In this setting, a natural extension of real-valued inference described above is the projection onto the $\esd$ first directions of any dimension reduction algorithm. The corresponding Mapper is now a multivariate Mapper and the underlying filter is the projection onto the $\esd$ first directions of the covariance operator of $P$. See Section~\ref{sec:kpca} above.
     \item {\bf Multivariate regression.}  Multivariate regression is the generalization of (univariate) regression when the variable $Y$ in Equation~\eqref{eq:reg} is now a random vector. 
     \item {\bf Multi-class classification.}  We observe a categorical variable $Y_i \in  \{0,\dots,k\}$ at each point $X_i$. Let $f_k(x) = P(Y=k|X=x)$ be the probability of class $k$ at $x \in \X$. The underlying filter is now the vector of estimated probabilities $f = (f_0,\dots, f_k)$, which can be obtained with classification methods in statistical machine learning. 
\end{itemize}

{\bf Synthetic example.} We now describe two multi-class classification problems and display the corresponding Mappers. In the first one, we generated a data set in two dimensions with three different classes which are entangled with each other. See Figure~\ref{fig:blobs} (left) for an illustration. We then trained a Random Forest classifier on this data set, and computed the estimated probabilities for each of the training points, meaning that we have an estimated multivariate filter $\hf:\R^2\rightarrow [0,1]^3$. The corresponding Mapper (computed with $10$ intervals and overlap $30\%$ for each class) is shown in Figure~\ref{fig:blobs} (right). Moreover, the Mapper nodes are colored with the variance of the class probability distributions: the smaller the variance, the more confident the prediction. It is clear from the Mapper that the classifier induces a topological stratification of the data, in the sense that points in the middle of the space (located in the middle of the triangle-shaped Mapper), on which the classifier is unsure, connect with points for which the classifier hesitates between two classes (located in the middle of the  ``edges" of the triangle), which themselves connect with points where the classifier is confident (located at the  ``corners" of the triangle), leading to some non-trivial 1-dimensional topological features in the data, which are not visible at first sight on the data set. We believe this visualization could be of great help when it comes to interpreting the output of standard statistical machine learning methods. \\

\begin{figure}[h]
    \centering
    \includegraphics[width=6cm]{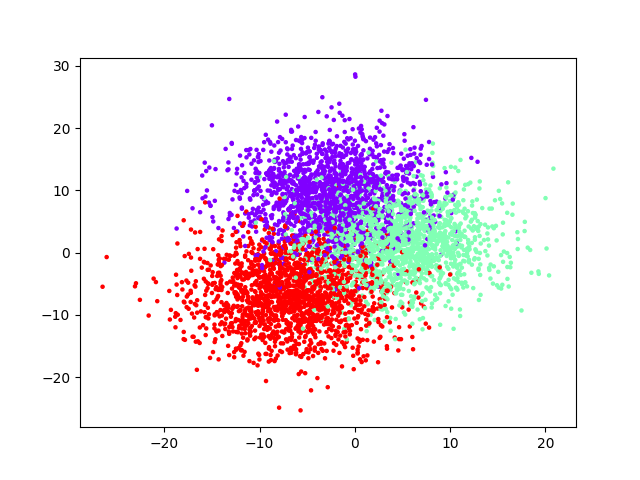}
    \includegraphics[width=6cm]{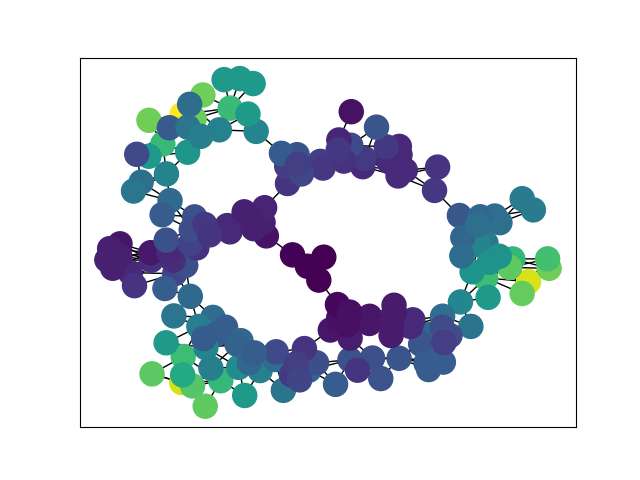}
    \caption{Three label classification problem and its corresponding Mapper. Left: we generate points in 2D with three different groups (red, purple, green). Right: Mapper computed with the posterior probability of a Random Forest classifier. Nodes are colored with the variance of the estimated probabilities from low (yellow) to large (dark blue).}
    \label{fig:blobs}
\end{figure}

{\bf Accelerometer data.} In our second example, we study a data set of time series obtained from accelerometers placed on people doing six  possible types of activities, namely  ``standing",  ``sitting",  ``laying",  ``walking",  ``walking upstairs" and  ``walking downstairs". From the raw data, $561$ features have been extracted from sliding window, see \cite{anguita2013public} and the data website\footnote{\url{https://archive.ics.uci.edu/ml/data sets/Human+Activity+Recognition+Using+Smartphones}} for more details. A Naive Bayes classifier has been trained on the $7,352$ observations. We finally generated an associated Mapper with the corresponding estimated probabilities (computed with $3$ intervals and $30\%$ gain for each class), and we colored the nodes with variance, similarly to what was done above. We show the Mapper, as well as representative time series for some of its nodes, in Figure~\ref{fig:accelero}. Again, the classifier is inducing a topological stratification of the data, with two connected components (corresponding to the two global types of activities, namely walking activities or stationary activities), which are themselves stratified into three activities connected by time series where the classifier is unsure.

\begin{figure}[h]
    \centering
    \includegraphics[width=10cm]{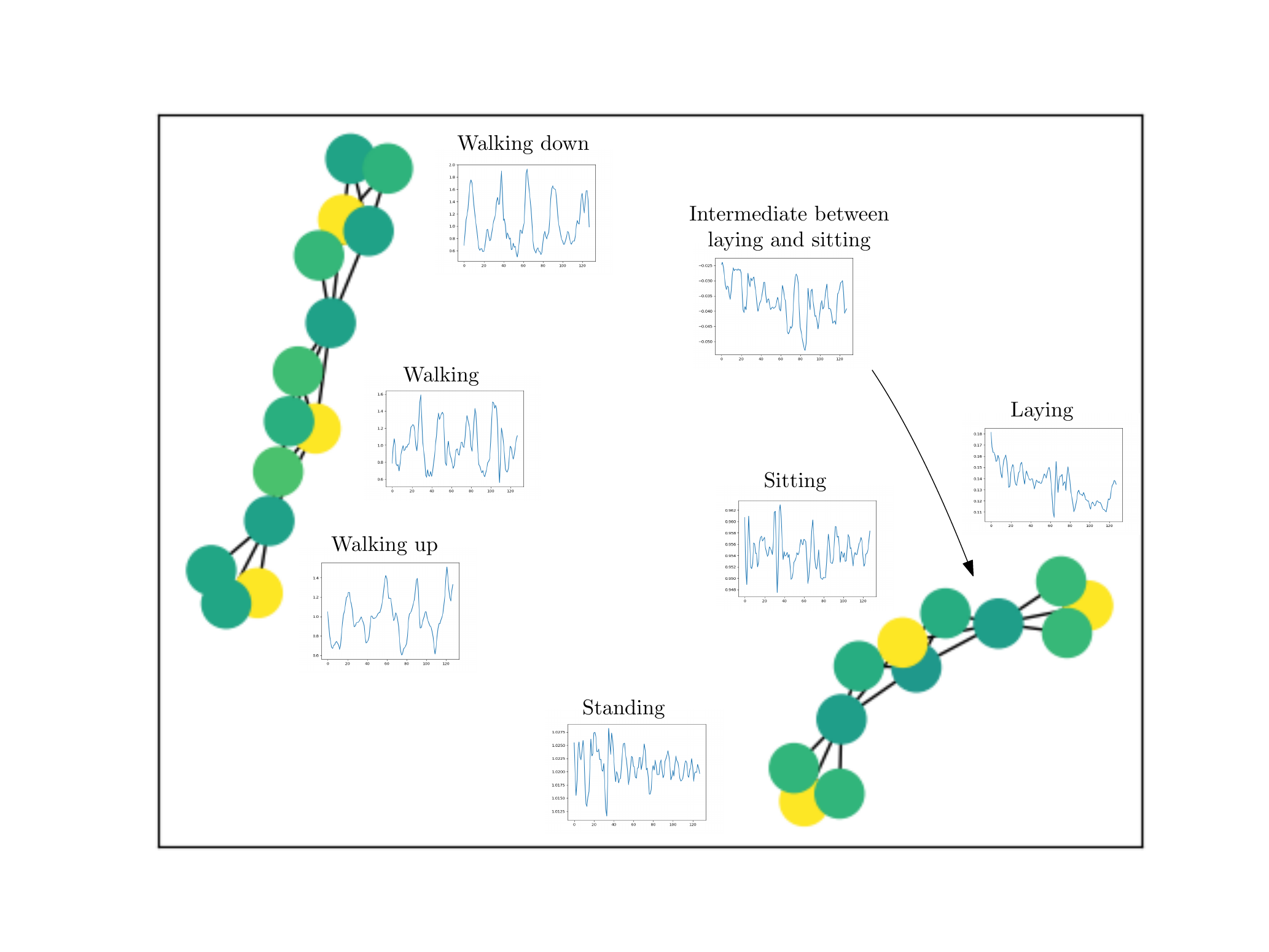}
    \caption{Mapper computed on accelerometer data with the posterior probability of a Naive Bayes classifier. Nodes are colored with the variance of the estimated probabilities from low (yellow) to large (dark green).}
    \label{fig:accelero}
\end{figure}


\subsection{Stochastic Filter with Conditional Probability Distributions}
\label{sec:single}

We now assume that we observe an i.i.d sample  $\{(X_i,Y_i)\,:\,1\leq i\leq n\}$, where $X_i \in \mathcal X$ and   $Y_i \in \R$. In this setting, we propose to consider the more complex filter function which is defined as the conditional distribution $(Y|X)$: the value of this filter at $x$ is the conditional distribution $(Y|X)$. In this framework, the filter domain is thus the space of probability distributions. 

In practice, it might be tempting to directly compute the standard Mapper with the $Y_i$'s as filter values. However, this approach does not really make sense because there is no reason for this Mapper to converge to a deterministic Reeb space for some underlying filter function. Having in mind that the relevant target filter is the conditional probability distribution $(Y|X)$, it is clear that this naive approach is not a good strategy for this aim, since single observations can be very poor estimates of the corresponding distributions.

\subsubsection{Mapper with probability distributions}

Let $\mathcal P$ be the set of probability measures on $\R$. For $x \in \X$, let $\mathcal \nu_x$ be the conditional distribution $(Y|X=x)$. Let $\nu$ be the filter $\nu:x \in \X \mapsto  \nu_x \in \mathcal P$.
Various metrics can be proposed on $\mathcal P$, one of them being the Prokhorov metric~\cite{billingsley2013convergence}, which metrizes weak convergence. 
Generally speaking, the Reeb space $\reeb_\nu(\X)$ is  difficult to infer since it requires to estimate the conditional probability distribution $\mathcal \nu_x$ for all points of $\X$, which is a difficult task, especially for high dimensional data---see for instance  \cite{efromovich2007conditional}. As far as we know, conditional density estimation on submanifolds has not been studied yet.
Moreover, as soon as $\nu$ is injective, which is not a strong assumption in practice, the Reeb space will be isomorphic to $\X$ and it will not provide more information than standard manifold learning procedures~\cite{ma2011manifold}. We thus propose to study approximations of $\reeb_\nu(\X)$, using a filter that is a simple descriptor (such as the mean or the histogram) of $\nu_x$. In this situation, from a data analysis perspective, crude approximations of the Reeb space  shows more interesting patterns than those provided by the Reeb space itself. \\

{\bf Mean- and histogram-based Mappers.} 
Let $\mathcal I =(I_1,\dots,I_d)$ be a partition of $\R$  with intervals. We define the histogram filter $\Hist$ associated to $\mathcal I$ by 
$\Hist_j(x)  = P \left(Y \in I_j  \left|  X = x   \right. \right) $ 
for $j =1, \dots, d$. 
The codomain of $\Hist$ is in $\R^d$, i.e., it is a multivariate filter, with corresponding Reeb space $\reeb_{\Hist}(\X)$. 
We then propose to compute the Mapper with an estimated histogram, which we call the {\em histogram-based Mapper}, using the Nadaraya-Watson kernel estimator: 
\begin{equation*}
\widehat{\Hist}_j(x) =   \frac{\sum_{i = 1, \dots, n}\mathds{1} _{Y_i \in I_j} K_h(X_i-x)}{\sum_{i = 1, \dots, n} K_h(X_i-x)}
\end{equation*}
where $K_h(x) = \frac 1h K (\frac xh)$ for a kernel function $K$, which we choose, in practice, to be the indicator function of the unit ball in the ambient Euclidean space. 

Note that a simpler approach is to estimate the (conditional) mean $f(x) =\E(Y | X=x)$, and we call the corresponding estimator the {\em mean-based Mapper}. However, as illustrated in numerical experiments presented below, it may be not sufficient to retrieve interesting data structure. \\

\subsubsection{Numerical experiments}

We now provide examples of computations of our Mapper-based estimators computed from single realizations of synthetic conditional probability distributions \footnote{Our code is freely available at \texttt{https://github.com/MathieuCarriere/metricmapper}}. We generate $5,000$ points from an annulus, and we looked at two conditional distributions for each point, namely Gaussians and bimodal ones. See Figure~\ref{fig:gaussian} and~\ref{fig:bimodal}. In each of these figures, we display five Mappers: the standard Mapper, the mean-based Mapper when the true conditional mean is supposed to be known, the mean-based Mapper when this mean is estimated, the histogram-based Mapper when the true histogram is supposed to be known, and the histogram-based Mapper when the histogram is estimated. We also plot, for the standard Mapper and the mean-based Mappers, a 3D embedding of the data set, with the mean values used as height. For the standard Mapper and the mean-based Mappers, we used an interval cover with $15$  intervals and overlap percentage $30\%$. For the histogram-based Mapper, we used histograms with $100$ bins and an $0.5$-thickening of a $K$-PDTM cover~\cite{brecheteau2020k} with $K=10$ cover elements. \\

\textbf{Gaussian conditional.} In Figure~\ref{fig:gaussian}, we generated Gaussian conditional probability distributions centered on the second coordinates of the points. It can be seen that the standard Mapper recovers the underlying structure, but in a very imprecise way, in the sense that the feature size is much smaller than it should be, due to the variances of the distributions that induce very noisy filter values. On the other hand, the mean-based Mappers and the histogram-based Mappers all recover the correct structure in much more precise fashion. \\

\begin{figure}[h]
    \centering
    \includegraphics[width=13cm]{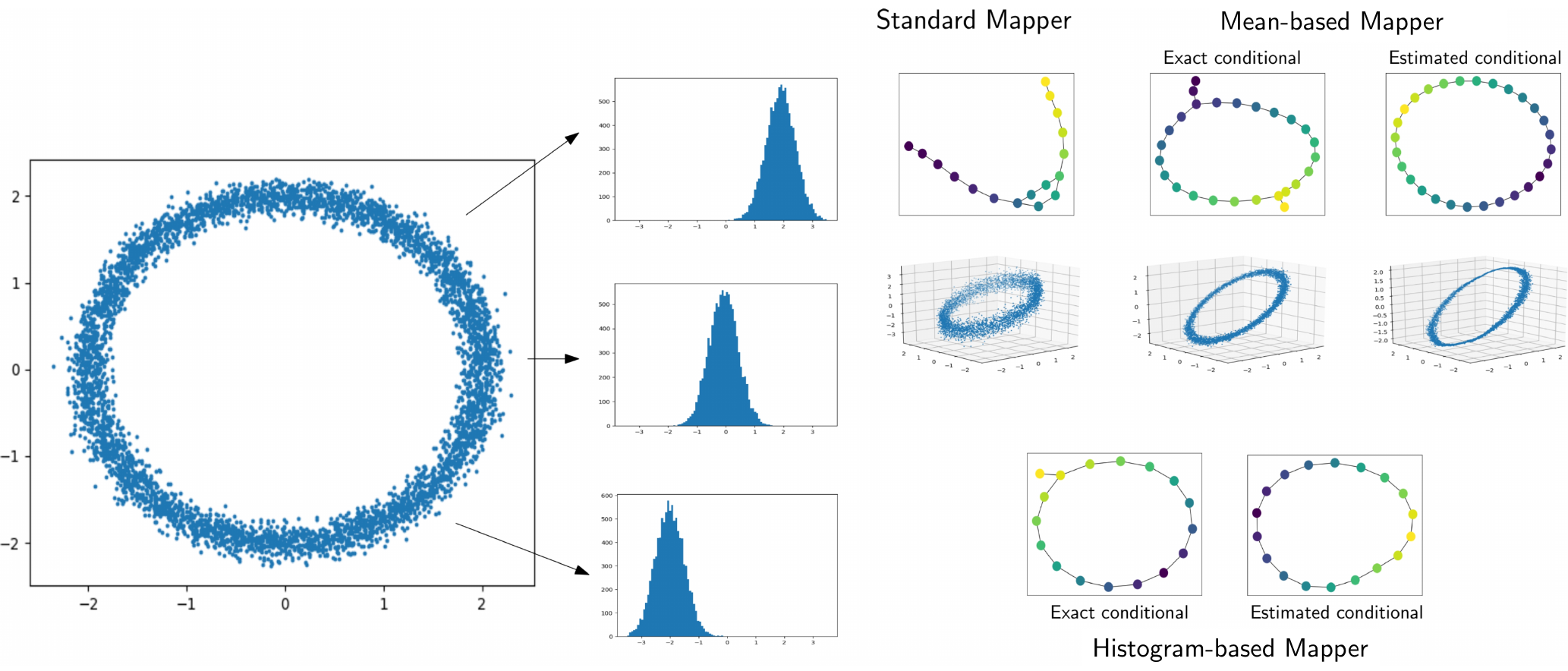}
    \caption{Standard, mean- and histogram-based Mappers computed with exact and estimated Gaussian conditional probability distributions. }
    \label{fig:gaussian}
\end{figure}

\textbf{Bimodal conditional.} In Figure~\ref{fig:bimodal}, we generate bimodal conditional probability distributions whose modes are centered on the second coordinate and its opposite (minus the minimum of the coordinates values). This way, all conditional probability distributions have the same mean. This time, the standard Mapper gets fooled by the probability distributions, and outputs two topological structures instead of one, due to the two modes of the distributions. The mean-based Mappers also fail due to the fact that the distributions all have the same mean, which mixes all points together and makes topological inference very difficult, leading to very noisy Mappers. On the other hand, the histogram-based Mappers both manage to retrieve the correct structure in a precise way. 

\begin{figure}[h]
    \centering
    \includegraphics[width=13cm]{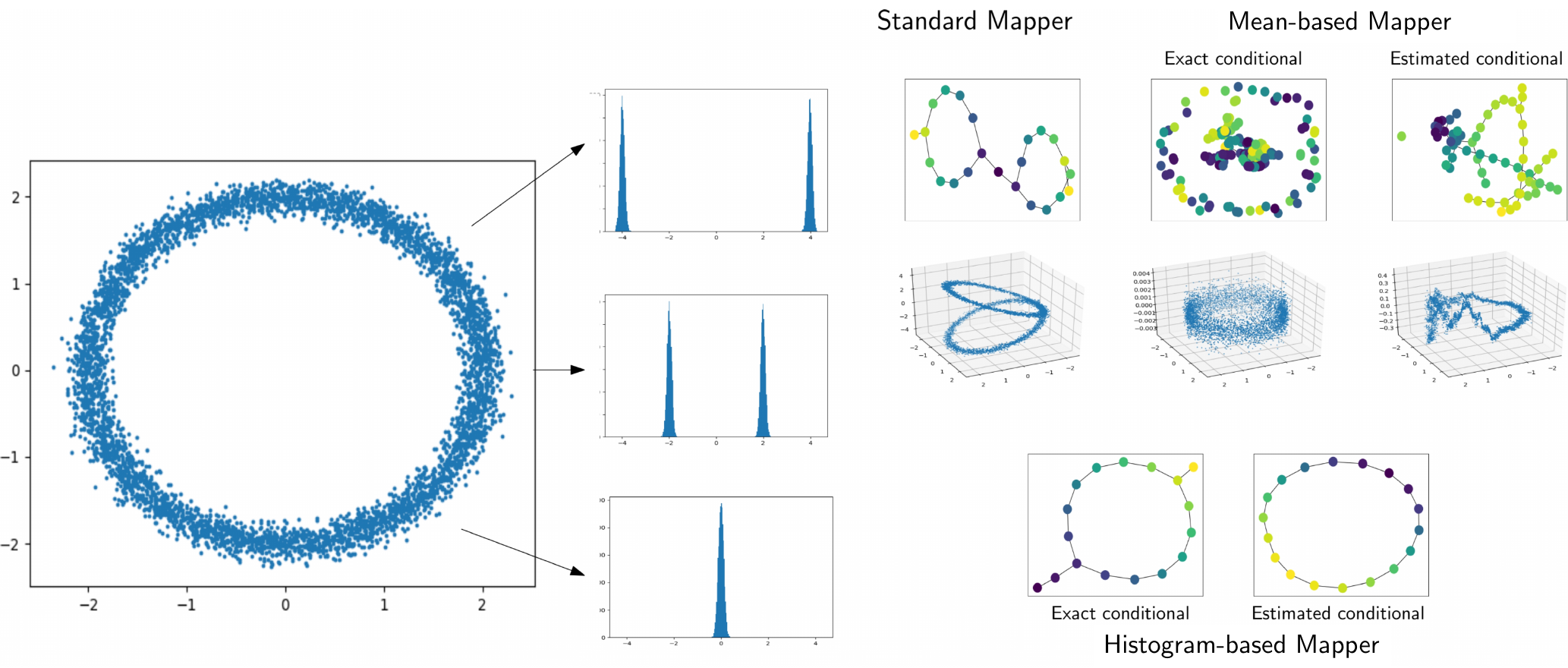}
    \caption{Standard, mean- and histogram-based Mappers computed with exact and estimated bimodal conditional probability distributions. }
    \label{fig:bimodal}
\end{figure}

\subsection{Stochastic filter with combinatorial graphs}\label{sec:metricex}

We end this application section by providing an example of our Mapper-based estimator, when the domain of the filter function is the space of combinatorial graphs. More specifically, we generated a graph for each data point of the annulus data set, using the Erd\H{o}s–R\'enyi model on 20 nodes, and using the first coordinate of the points (normalized between 0 and 1) as the model parameter (that is, any possible edge among the 20 nodes appears with probability given by the model parameter). This means that points located at the bottom of the annulus will have graphs with fewer edges than those above. See Figure~\ref{fig:mappgraph} (left). Then, we used the graph edit distance (provided in the \texttt{networkx} Python package) and a Voronoi cover with $10$ cells (corresponding to $10$ randomly sampled germs) and $0.5$-thickening to compute our estimator. The corresponding sMapper is shown in Figure~\ref{fig:mappgraph} (right). One can see that the correct topology is retrieved by our estimator.

\begin{figure}[h]
    \centering
    \includegraphics[width=12cm]{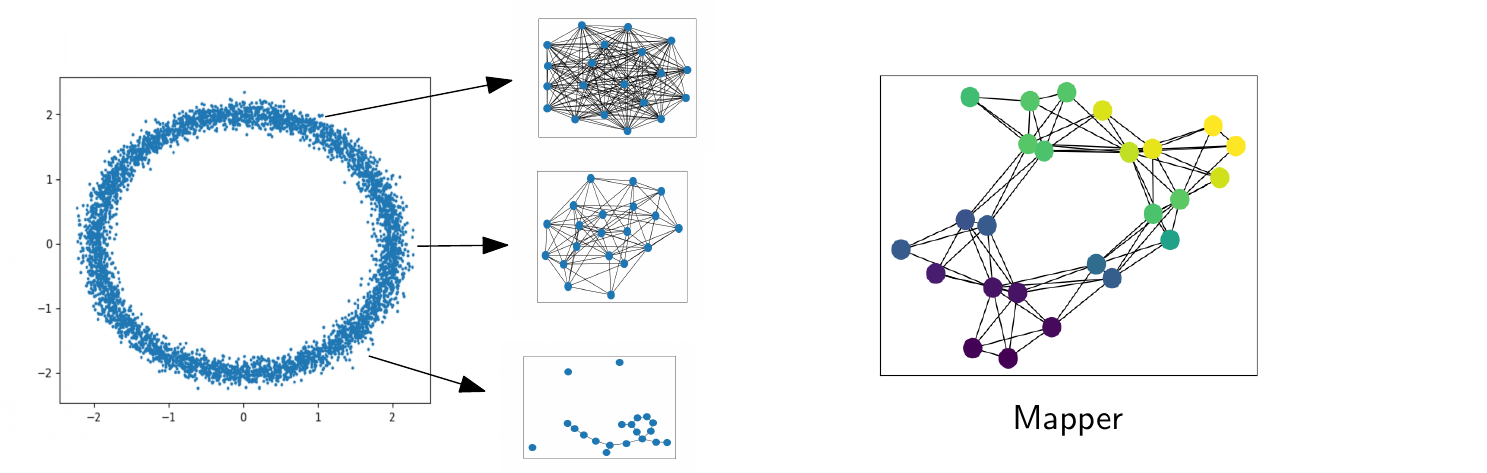}
    \caption{Example of Mapper computation for combinatorial graphs.}
    \label{fig:mappgraph}
\end{figure}

\section{Conclusion and future directions}\label{sec:conc}

 In this article, we presented a computable Mapper-based estimator that enjoys statistical guarantees for its approximation of its corresponding target Reeb space. Moreover, we demonstrated how it can be applied when the filter is estimated from a random sample of data, which we call the Stochastic Filter setting. In this case, we demonstrated a few applications in statistical machine learning, and we provided examples in which the usual Mapper fails dramatically, whereas our estimators still succeed. Much work is still needed for future directions, including demonstrating optimality and stability of the estimator. Moreover, we plan on adapting bootstrap methods to compute and interpret confidence regions. We also plan to adapt specific clustering algorithms in the space of distributions to propose efficient covers in this setting.  In the longer term, we also plan to strengthen our results by extending them to the interleaving distance of~\cite{Munch2016}.
  
\paragraph{Acknowledgements.} The authors would like to thank Claire Br\'echeteau and Cl\'ement Levrard for helpful discussions on the control of the resolution of the $k$-means algorithm, and Yusu Wang for suggesting the use of filter-based pseudometrics. 

\paragraph{Conflict of interest.}
On behalf of all authors, the corresponding author states that there is no conflict of interest.


\bibliography{biblio}

\appendix

\section{Proofs}

\subsection{Proof of Theorem \ref{th:stochmapper}}\label{app:bound}

We assume that (H1), (H2), (H3) and (H4) of Section~\ref{sub:calib_risk} are satisfied. The parameters $s\geq s_n$ and $\delta_n$ are assumed to be chosen according to \eqref{eq:deltan} and \eqref{eq:sn}. Recall that the point cloud $\hX_{n,s}$ is a refinement of the point cloud $\hX_{n}$, as defined in Section \ref{sec:stochmapper}. We also introduce the generalized inverse of a modulus of continuity $\omega$:
$$ 
\omega^{-1} (v) = \left\{ u  \, : \,   \omega(u) \geq v \right\}.
$$

\medskip

\textbf{Approximation Lemmata.}  We first prove three approximations. 
In the first one, we show that our estimator 
$\mapper_n$ is actually equivalent to the (continuous) Mapper of an associated neighborhood graph. 

\begin{lemma}\label{lem:approxmapper}
The Mappers $\mapper_n$ and $\mapper_{\hfpl,\U}(G_{\delta_n})$ are isomorphic as simplicial complexes. Hence,
\begin{equation*}\label{eq:approxmapper}
    d_{\rm GH}\left((\mapper_n, \tilde d_{\hfpl,\U}),(\mapper_{\hfpl,\U}(G_{\delta_n}),\tilde d_{\hfpl,\U})\right)=0
\end{equation*}
\end{lemma}

\begin{proof}  
Let $U_\alpha \in \U$, and $\mathcal C_\alpha$ be a connected component of $
\hfpl^{-1}(U_\alpha)$ in $\Gn$. 
We claim that $\mathcal C_\alpha\cap \hX_{n,s}\neq\emptyset$. Indeed, if we assume that $\mathcal C_\alpha\cap \hX_{n,s}=\emptyset$ then it means that $\mathcal C_\alpha$ is constituted from a subpath $\bar e$ of an edge $e$ of $G_{\delta_n}$, that does not contain the endpoints of $e$ in $\hX_n$, nor any points of $\hX_{n,s}$ in the subdivision of $e$. 
Hence, $e$ is element-crossing and $U_\alpha$ is crossed by $e$.
By definition, the length $|\hfpl(e)\cap U_\alpha |$ must be at least $\ell(\hX_n,\hf,\U)$. Moreover, due to the subdivision process, the length $| \bar e |$ must be less than $\delta_n/(1+s) \leq \delta_n/(1+s_n)  $, meaning that $|\hfpl(\bar e)| = |\hfpl(e) \cap U_\alpha|$ must be less than $\hat{\omega}_{\rm PG}(\delta_n/ (1+s_n) )$. Hence, using the definition of $s_n$, we have the following inequalities:
$$\frac{\ell}{2} \geq \hat{\omega}_{\rm PG}\left(\frac{\delta_n}{1+\left\lfloor\delta_n/\hat{\omega}_{\rm PG}^{-1}(\ell/2)\right\rfloor}\right)\geq |\hfpl(e) \cap U_\alpha| \geq \ell,$$
which leads to a contradiction (except for $\ell=0$, which happens with null probability).

Hence, for each $U_\alpha$ and connected component $\mathcal C_\alpha$ of $
\hfpl^{-1}(U_\alpha)$ in $\Gn$, 
there is one point of $\hX_{n,s_n}$ that belongs to $\mathcal C_\alpha
$. Now, let $\tilde{\mathcal C}_\alpha$ be the connected component in $G_{\delta_n,s_n}(U_\alpha)$ (see Equation~(\ref{eq:subgraph}))
associated to this point. 
We now claim that $\tilde{\mathcal C}_\alpha$ is included in $\mathcal C_\alpha$.
Indeed, since $G_{\delta_n,s_n}$ is nothing but a subdivision of $G_{\delta_n}$, and since any edge of $G_{\delta_n,s_n}$ in $\tilde{\mathcal C}_\alpha$  must also be present in ${\mathcal C}_\alpha$ (otherwise it would induce an element-crossing edge in $G_{\delta_n}$ whose intersection with the corresponding crossed cover element would contain no points in $\hX_{n,s_n}$, which is impossible for the reason mentioned above), it follows that $\mathcal{C}_\alpha$ deform-retracts on $\tilde{\mathcal{C}}_\alpha$. Hence,
$\mapper_n$ and $\mapper_{\hfpl,\U}(\Gn)$ have the exact same sets of nodes. 

The same argument applies straightforwardly to show that the connected components in the intersections are also in bijection, which means that the simplices of both Mappers are in correspondence as well.
\end{proof}

Let $d_g$ denote the geodesic distance on $\X$.
Let $d_{{\rm H},n} =d^g_{\rm H}(\hX_n, \X)$ where we denote by $d^g_{\rm H}$ the Hausdorff distance 
computed with geodesic distances. 
\begin{lemma}\label{lem:approxdf}
Let $x,x'\in G_\delta$. Then, $|d_{\fpl}(x,x')-d_{\fpl}(\zeta(x), \zeta(x'))|\leq 2\omega(\delta)$, where $\zeta(x)\in\hX_n$ is the closest endpoint of the edge to which $x$ belongs if $x\not\in\hX_n$ and $x$ otherwise. 

Similarly, let $x,x'\in\X$. Then, $|d_f(x,x')-d_f(\zeta(x), \zeta(x'))|\leq 2\omega(d_{{\rm H},n})$, where $\zeta(x)\in\hX_n$ is such that $d_g(x,\zeta(x))\leq d_{{\rm H},n}$ (whose existence is guaranteed with $d^g_{\rm H}(\hX_n,\X)=d_{{\rm H},n}$). 
\end{lemma}

\begin{proof}
Let $\gamma$ be a path going from $\zeta(x)$ to $\zeta(x')$ achieving\footnote{We assume for sake of simplicity in this proof that 
the infimums in the definition of the filter-based pseudometric are always achieved by some path. However, our proof extends straightforwardly to the general case by considering limits of sequences of paths converging to the infimum. } $d_{\fpl}(\zeta(x), \zeta(x'))$. Let $\gamma'=\gamma_{x'}\circ\gamma\circ\gamma_x$, where $\gamma_x$ is the path going from $x$ to $\zeta(x)$ along the edge $e_x$ to which $x$ belongs, and $\gamma_{x'}$ is the path going from $\zeta(x')$ to $x'$ along the edge $e_{x'}$ to which $x'$ belongs. Also, let $\tilde\zeta(x)\in\hX_n$ denote the other endpoint of $e_x$, and similarly for $\tilde\zeta(x')$. Then $\gamma'$ is a path from $x$ to $x'$. 

Now, $\fpl\circ\gamma'=\fpl\circ\gamma_{x'} \cup \fpl\circ\gamma \cup \fpl\circ\gamma_x\subseteq \fpl\circ e_{x'} \cup \fpl\circ\gamma \cup \fpl\circ e_x$, and 
\begin{align*}
{\rm diam}_{\Z}(\fpl\circ\gamma') & \leq {\rm diam}_{\Z}(\fpl\circ e_{x'} \cup \fpl\circ\gamma \cup \fpl\circ e_x) \\
& \leq{\rm max} \{d_\Z(f(u),f(v))\,:\,u,v\in \{\tilde\zeta(x),\tilde\zeta(x')\}\cup (\gamma\cap\hX_n)\}\\
& \leq 
{\rm max} \{d_\Z(f(u),f(v))\,:\,u,v\in  (\gamma\cap\hX_n)\} + d_\Z(f\circ\zeta(x), f\circ\tilde\zeta(x)) + d_\Z(f\circ\zeta(x'), f\circ\tilde\zeta(x')) \\
& = d_{\fpl}(\zeta(x),\zeta(x')) + d_\Z(f\circ\zeta(x), f\circ\tilde\zeta(x)) + d_\Z(f\circ\zeta(x'), f\circ\tilde\zeta(x')) \\
& \leq d_{\fpl}(\zeta(x),\zeta(x')) + 2\omega(\delta)
\end{align*}
Hence $d_{\fpl}(x,x')\leq d_{\fpl}(\zeta(x),\zeta(x')) + 2\omega(\delta)$.

Now, assume $d_{\fpl}(x,x') < d_{\fpl}(\zeta(x),\zeta(x')) - 2\omega(\delta)$, and let $\gamma$ be a path from $x$ to $x'$ achieving $d_{\fpl}(x,x')$. Again, let $\gamma'=\gamma_{x'}\circ \gamma \circ \gamma_x$, where $\gamma_x$ is the path going from $\zeta(x)$ to $x$ along the edge $e_x$ to which $x$ belongs, and $\gamma_{x'}$ is the path going from $x'$ to $\zeta(x')$ along the edge $e_{x'}$ to which $x'$ belongs. Then:
\begin{align*}
{\rm diam}_{\Z}(\fpl\circ\gamma') & \leq {\rm diam}_{\Z}(\fpl\circ \gamma_{x'} \cup \fpl\circ\gamma \cup \fpl\circ \gamma_x) \\
& \leq d_{\fpl}(x,x') + d_\Z(f(x'), f\circ\zeta(x')) + d_\Z(f(x), f\circ\zeta(x)) \\
& \leq d_{\fpl}(x,x') + 2\omega(\delta)\\
& < d_{\fpl}(\zeta(x),\zeta(x')),
\end{align*}
which is impossible since $d_{\fpl}(\zeta(x),\zeta(x')) \leq {\rm diam}_{\Z}(\fpl\circ\gamma')$.

The proof for the second statement is exactly the same.
\end{proof}

In our third lemma, we show that the Reeb space of a space and its neighborhood graph approximation are actually close, 
provided that the graph is built on top of a dense enough point cloud.

\begin{lemma}\label{lem:approxreeb}  Assume $6d_{{\rm H},n}\leq \delta_n \leq 2\cdot\reach(\X)$. Then, one has
\begin{equation*}\label{eq:approxreeb}
    d_{\rm GH}\left((\reeb_{\hfpl}(\Gn), \tilde d_{\hfpl}),(\reeb_{f}(\X), \tilde d_f)\right)
    \leq  4\omega(2\delta_n) + 2 \| \fpl-\hfpl \|_\infty .
\end{equation*}
\end{lemma}

\begin{proof}
By the triangle inequality, we have:
\begin{align}
    d_{\rm GH} & \left((\reeb_{\hfpl}(\Gn), \tilde d_{\hfpl}),(\reeb_{f}(\X), \tilde d_f)\right) \nonumber \\ 
    & \leq d_{\rm GH}\left((\reeb_{\hfpl}(\Gn), \tilde d_{\hfpl}), (\Gn, d_{\hfpl})\right)
    + d_{\rm GH}\left((\Gn, d_{\hfpl}),(\X, d_f)\right)
    + d_{\rm GH}\left((\X, d_f), (\reeb_{f}(\X), \tilde d_f)\right)\nonumber \\
    & = d_{\rm GH}\left((\Gn, d_{\hfpl}),(\X, d_f)\right)\text{ by Proposition~\ref{prop:spacereeb}}  \nonumber \\
    & \leq d_{\rm GH}\left((\Gn, d_{\hfpl}),(\Gn, d_{\fpl})\right)
    + d_{\rm GH}\left((\Gn, d_{\fpl}), (\X, d_{f})\right) \nonumber
\end{align}

{\bf First term.} Let us bound $d_{\rm GH}((\Gn, d_{\hfpl}),(\Gn, d_{\fpl}))$. The identity map $\iota:\Gn\rightarrow \Gn$ can be used to define a correspondence, with which it is clear that: $$|d_{\fpl}(x,x')-d_{\hfpl}(\iota(x),\iota(x'))|=|d_{\fpl}(x,x')-d_{\hfpl}(x,x')|\leq 2\|\fpl-\hfpl\|_\infty.$$ 

Indeed, one has $d_{\hfpl}(x,x')\leq {\rm diam}_\Z(\hfpl\circ\gamma)$, where $\gamma$ is a path achieving $d_{\fpl}(x,x')$. Thus,
$d_{\hfpl}(x,x')\leq {\rm diam}_\Z(\hfpl\circ\gamma)\leq{\rm diam}_\Z(\fpl\circ\gamma) + 2\|\fpl-\hfpl\|_\infty=d_{\fpl}(x,x')+2\|\fpl-\hfpl\|_\infty$. Symmetrically, one can also show that $d_{\fpl}(x,x')\leq d_{\hfpl}(x,x')+2\|\fpl-\hfpl\|_\infty$, hence the result. \\


{\bf Second term.} Let us now bound $ d_{\rm GH}((\Gn, d_{\fpl}), (\X, d_{f}))$. Let $\mathcal C$ be a correspondence between $\X$ and $\Gn$ defined with $\mathcal C=\{(x,\zeta(x))\,:\,x\in\X\}\cup \{(\zeta(y),y)\,:\,y\in\Gn\}$ (see Lemma~\ref{lem:approxdf}). \\

{\bf Restriction to point cloud.} First,  we show that we can restrict to pairs of points in $\hX_n$, up to some constant. Indeed, it follows from Lemma~\ref{lem:approxdf} that, for all $x,x'\in\X$ and $y,y'\in\Gn$:
\begin{align*}
    |d_f(x,x')-d_{\fpl}(\zeta(x),\zeta(x'))| & \leq |d_f(x,x')-d_f(\zeta(x),\zeta(x'))| + |d_f(\zeta(x),\zeta(x'))-d_{\fpl}(\zeta(x),\zeta(x'))| \\
    & \leq 2\omega(d_{{\rm H},n}) + |d_f(\zeta(x),\zeta(x'))-d_{\fpl}(\zeta(x),\zeta(x'))|\\
    |d_f(\zeta(y),\zeta(y'))-d_{\fpl}(y,y')| & \leq |d_{\fpl}(y,y')-d_{\fpl}(\zeta(y),\zeta(y'))| + |d_f(\zeta(y),\zeta(y'))-d_{\fpl}(\zeta(y),\zeta(y'))| \\
    & \leq 2\omega(\delta_n) + |d_f(\zeta(y),\zeta(y'))-d_{\fpl}(\zeta(y),\zeta(y'))| \\
    |d_f(x,\zeta(y'))-d_{\fpl}(\zeta(x),y')| & \leq |d_f(x,\zeta(y'))-d_f(\zeta(x),\zeta(y'))| + |d_{\fpl}(\zeta(x),y')-d_{\fpl}(\zeta(x),\zeta(y'))| \\
    & + |d_f(\zeta(x),\zeta(y'))-d_{\fpl}(\zeta(x),\zeta(y'))|\\
    & \leq \omega(\delta_n)+ \omega(d_{{\rm H},n}) + |d_f(\zeta(x),\zeta(y'))-d_{\fpl}(\zeta(x),\zeta(y'))|
\end{align*}
Since $\omega(d_{{\rm H},n})\leq\omega(\delta_n)$, one has $d_{\rm GH}((\Gn, d_{\fpl}), (\X, d_{f}))\leq 2\omega(\delta_n) + {\rm max}_{x,x'\in\hX_n}\ |d_f(x,x')-d_{\fpl}(x,x')|$. \\

Let $x,x'\in\hX_n$. We now find upper and lower bounds for $d_{\fpl}(x,x')-d_f(x,x')$. \\

{\bf Upper bound.} In order to upper bound $d_{\fpl}(x,x')-d_f(x,x')$, we first show that $d_{\fpl}(x,x')$ cannot be arbitrarily large relative to $d_f(x,x')$.
 Let $\gamma$ be a path on $\X$ from $x$ to $x'$ achieving $d_f(x,x')$. Since $d^g_{\rm H}(\hX_n,\X)\leq d_{{\rm H},n}$, for each $t\in[0,1]$, there exists $x_t\in\hX_n$ such that $d_g(\gamma(t),x_t)\leq d_{{\rm H},n}$. Moreover, since $\hX_n$ is finite, the set $\{x_t\,:\,t\in[0,1]\}$ can be written as
$\{x_{t_1},\dots,x_{t_m}\}$ for some $m\in\N^*$, with $t_1\leq \dots \leq t_m$. Moreover, we claim that $\|x_{t_i}-x_{t_{i+1}}\|\leq \delta_n$, i.e., the set $\{x,x_{t_1},\dots,x_{t_m},x'\}$ forms a path in $\Gn$. Indeed:
\begin{align*}
    \|x_{t_i}-x_{t_{i+1}}\| & \leq \|x_{t_i}-\gamma(t_i)\| + \|\gamma(t_i) -\gamma(t_{i+1})\| + \|\gamma(t_{i+1}) - x_{t_{i+1}}\| \\
    & \leq d_g(x_{t_i},\gamma(t_i)) + d_g(\gamma(t_i) ,\gamma(t_{i+1})) + d_g(\gamma(t_{i+1}) , x_{t_{i+1}}) \\
    & \leq 2d_{{\rm H},n} +  d_g(\gamma(t_i) ,\gamma(t_{i+1}))
\end{align*}
The geodesic distance $d_g(\gamma(t_i) ,\gamma(t_{i+1}))$ is necessarily less than $4d_{{\rm H},n}$, otherwise it would
be possible to find a point along the geodesic, say $\gamma(\bar t)$, such that $t_i \leq \bar t\leq t_{i+1}$ and $d_g(\gamma(\bar t),\gamma(t_i))> 2d_{{\rm H},n}$ and $d_g(\gamma(\bar t), \gamma(t_{i+1}))>2d_{{\rm H},n}$, which lead to $d_g(\gamma(\bar t),x_{t_i})>d_{{\rm H},n}$ and $d_g(\gamma(\bar t),x_{t_{i+1}})>d_{{\rm H},n}$, contradicting $d^g_{\rm H}(\hX_n,\X)\leq d_{{\rm H},n}$. Hence, $\|x_{t_i}-x_{t_{i+1}}\|\leq 6d_{{\rm H},n}\leq \delta_n$ by assumption. Let $\gamma'$ be the path from $x$ to $x'$ in $\Gn$ that goes through the points $\{x,x_{t_1},\dots,x_{t_m},x'\}\in\hX_n$. We also use $x_0$ and $x_1$ to denote $x$ and $x'$.
Then 

\begin{align*}
d_{\fpl}(x,x') & \leq {\rm diam}_{\Z}(\fpl\circ \gamma') \\
& \leq{\rm max}\ \{d_\Z(f(u),f(v))\,:\,u,v\in\{x,x_{t_1},\dots,x_{t_m},x'\})\}\\
&\leq d_f(x,x') + 2 \cdot {\rm max}\ \{d_\Z(f(x_t), f(\gamma(t)))\,:\,t\in\{0,t_1\dots,t_m,1\}\} \\
&\leq d_f(x,x')+ 2\omega(d_{{\rm H},n}) \leq  d_f(x,x')+ 2\omega(\delta_n)
\end{align*}

{\bf Lower bound.} Finally, we now show that $d_{\fpl}(x,x')$ cannot be arbitrarily small relative to $d_f(x,x')$. 
Let $\gamma$ be a path in $\Gn$ achieving $d_{\fpl}(x,x')$. Let $\gamma\cap \hX_n = \{x_0,x_1,\dots,x_m,x_{m+1}\}$, i.e., $\gamma$ goes through the points $x_0,\dots,x_{m+1}\in\hX_n$ with $x_0=x$ and $x_{m+1}=x'$. Finally, let $\gamma'$ be the path from $x$ to $x'$ in $\X$ defined with $\gamma'=\gamma_{m}\circ \dots \circ \gamma_0$, where $\gamma_i$ is a path achieving $d_g(x_i, x_{i+1})$.

Now, we claim that $\gamma'\subseteq \bigcup_{0\leq i \leq m+1} B_g(x_i,(\pi/2)\delta_n)$. 
Indeed, it follows from 
Lemma 3 in~\cite{Boissonnat2019b}
that $\|x_i-x_{i+1}\|\leq\delta_n \leq 2\cdot\reach(\X) \Rightarrow d_g(x_i,x_{i+1})\leq 2\cdot \reach(\X)\cdot{\rm arcsin}\left(\frac{\|x_i-x_{i+1}\|}{2\cdot\reach(\X)}\right) \leq (\pi/2) \|x_i-x_{i+1}\| \leq (\pi/2) \delta_n$. 
Then, one has
\begin{align*}
    d_f(x,x') & \leq {\rm diam}_\Z(f\circ\gamma') \\
    & \leq {\rm diam}_\Z(f(\bigcup\nolimits_{0\leq i \leq m+1} B_g(x_i,(\pi/2)\delta_n )))
    \\
    & = {\rm sup}\ \{d_\Z(f(u), f(v))\,:\,u,v\in \bigcup\nolimits_{0\leq i \leq m+1} B_g(x_i,(\pi/2)\delta_n
    )\}\\
    & \leq {\rm sup}\ \{d_\Z(f(u), f(v))\,:\,u,v\in \{x_0,\dots,x_{m+1}\}\} + 2\cdot {\rm max}_i\ {\rm diam}_\Z(f(B_g(x_i,(\pi/2)\delta_n
    )))\\ 
    & \leq d_{\fpl}(x,x') + 2\omega((\pi/2)\delta_n
    )
\end{align*}

We can finally conclude: $d_{\rm GH}((\Gn, d_{\fpl}), (\X, d_{f}))\leq 2\omega(\delta_n) + {\rm max}_{x,x'\in\hX_n}\ |d_f(x,x')-d_{\fpl}(x,x')| \leq 2\omega(\delta_n) + 2\omega((\pi/2)\delta_n
) \leq 4\omega(2\delta_n) $.
\end{proof}

We are now ready to prove Theorem~\ref{th:stochmapper}.

\begin{proof} {\bf Theorem~\ref{th:stochmapper}.} 
We first decompose the objective into three terms:
\begin{align}
    \E & \left[ d_{\rm GH} ((\mapper_{n},\tilde d_{\hfpl,\U}), (\reeb_f(\X),\tilde d_f)) \right] \nonumber \\
    & = \E \left[
    d_{\rm GH}((\mapper_{\hfpl,\U}(\Gn),\tilde d_{\hfpl,\U}), (\reeb_f(\X),\tilde d_f))
    \right] \text{ by Lemma~\ref{lem:approxmapper}}\nonumber \\
    & \leq \E\left[
    d_{\rm GH}((\mapper_{\hfpl,\U}(\Gn),\tilde d_{\hfpl,\U}), (\reeb_{\hfpl}(\Gn),\tilde d_{\hfpl}))
    \cdot \mathds{1}_\Omega
    \right] \label{eq2} \\
    & + \E\left[
    d_{\rm GH}((\reeb_{\hfpl}(\Gn),\tilde d_{\hfpl}), (\reeb_f(\X), \tilde d_f))
    \cdot \mathds{1}_\Omega
    \right] \label{eq3} \\
    & + \proba(\Omega^c)\cdot \omega(D_{\X}), \nonumber 
\end{align}

where $\Omega$ is the event $\{d_{{\rm H},n}\leq\delta_n/6\}\cap\{\delta_n\leq 2\cdot\reach(\X)\}$, and $D_\X$ is the diameter of $\X$.

Let us now bound (\ref{eq2}) and (\ref{eq3}):

\begin{itemize}
    \item Term (\ref{eq2}). According to Theorem~\ref{thm:convdgh}, we have 
    $$\E\left[
    d_{\rm GH}((\mapper_{\hfpl,\U}(\Gn),\tilde d_{\hfpl,\U}), (\reeb_{\hfpl}(\Gn),\tilde d_{\hfpl}))
    \right]
    \leq 5\cdot\E\left[\res(\U,\hfpl)\right].$$
    \item Term (\ref{eq3}). According to Lemma~\ref{lem:approxreeb}, we have:
    $$\E\left[
    d_{\rm GH}((\reeb_{\hfpl}(\Gn),\tilde d_{\hfpl}), (\reeb_f(\X), \tilde d_f))
    \right]\leq 4\E\left[\omega(2\delta_n)\right] + 2\|\fpl-\hfpl\|_\infty.$$ 

\end{itemize}
We conclude with Lemma~\ref{lem:EspOmega}.
\end{proof}

\begin{lemma} \label{lem:EspOmega}
Under assumptions {\rm (H1)}, {\rm (H2)} and {\rm (H3)}, and for $\Omega$ defined as before, one has
$$\proba(\Omega^c)\cdot  \omega(D_\X)  + 4\E[\omega(2\delta_n)] \leq C \omega\left(C'\frac{{\rm log}(n)^{(2+\beta)/b}}{n^{1/b}}\right)  $$
where $C,C'$ only depends on $a$, $b$ and on the geometric parameters of $\X$.
\end{lemma}
\begin{proof}
The proof is borrowed from Appendix A.7 in~\cite{Carriere2018a}, see this reference for more details on the proof.
Let $K = 2\cdot \reach(\X)$. 
Note that by definition: $$\proba(\Omega^c) +  4\E[\omega(2\delta_n)]\leq \proba(\egn>\delta_n/6) + \proba(\delta_n>K) +   4\E[\omega(2\delta_n)].$$ Moreover, since $P$ is $(a,b)$-standard, one has:
\begin{equation}\label{eq:cuevas}
    \proba(\deh(\hX_n,\X) \geq u ) \leq \min\left\{1, \frac{4^b}{a u^b}{\rm e}^{-a\left(\frac{u}{2}\right)^b n}\right\}=f_{a,b}(n,u), \forall u > 0.
\end{equation}

Let us bound each term independently. 

\paragraph*{Second term.} We have the following inequalities:
\begin{align*}
    \proba(\delta_n>K) & = \proba(d^E_{\rm H}(\hX_{s(n)},\hX_n)>K)\\
    & \leq \proba(d^E_{\rm H}(\hX_{s(n)},\X) + d^E_{\rm H}(\hX_{n},\X) >K) \\
    & \leq \proba(d^E_{\rm H}(\hX_{s(n)},\X) > K/2\ \cup\  d^E_{\rm H}(\hX_{n},\X) >K/2) \\
    & \leq \proba(d^E_{\rm H}(\hX_{s(n)},\X) > K/2)+  \proba(d^E_{\rm H}(\hX_{n},\X) >K/2) \\
    & \leq f_{a,b}(s(n),K/2) + f_{a,b}(n,K/2)
\end{align*} 

\paragraph*{First term }(See term $(B)$ in the proof of Proposition 13 in   \cite{Carriere2018a} for more details).  Note that when $\deh(\hX_n,\X) \leq K$, it follows from Lemma 3 in \cite{Boissonnat2019b} that $\egn\leq (\pi/2)\deh(\hX_n,\X)$. Thus,  we have the following inequalities:
\begin{align*}
     \proba(\egn>\delta_n/6) & \leq \proba(\egn>\delta_n/6\ \cap\  d^E_{\rm H}(\hX_n,\X) \leq K) + \proba(\deh(\hX_n,\X) > K) \\
     & \leq \proba(\deh(\hX_n,\X)>\delta_n/(3\pi)\ \cap\  d^E_{\rm H}(\hX_n,\X) \leq K) + \proba(\deh(\hX_n,\X) > K) \\
     & \leq \proba(\deh(\hX_n,\X)>\delta_n/(3\pi)) + \proba(\deh(\hX_n,\X) > K) \\
     & \leq \frac{2^{b-1}}{n{\rm log}(n)} + f_{a,b}(n,K)\text{ for $n$ large enough},
\end{align*}
since it is known that, given a constant $C>0$, the probability $\proba(\deh(\hX_n,\X) > C \delta_n)$ is always upper bounded by $\frac{2^{b-1}}{n{\rm log}(n)}$ for $n$ large enough (with the minimal required value for $n$ increasing with the constant $C$ and the ambient dimension).

\paragraph*{Third term}(See term $(A)$ in the proof of Proposition 13 in   \cite{Carriere2018a} for more details). This is the dominating term. Let $\bar D=\omega(2D_\X)$. Then, we have:
\begin{align*}
    \E[\omega(2\delta_n)] &= \int_0^{\bar D} \proba(\omega(2\delta_n) \geq\alpha){\rm d}\alpha \\
    &\leq \int_0^{\bar D} \proba\left(\deh(\hX_n,\X)\geq \frac{1}{4}\omega^{-1}(\alpha)\right){\rm d}\alpha + \int_0^{\bar D} \proba\left(\deh(\hX_{s(n)},\X)\geq \frac{1}{4}\omega^{-1}(\alpha)\right){\rm d}\alpha \\
    &\leq C'' \omega\left[\left(C'\frac{{\rm log}(s(n))}{s(n)}\right)^{1/b}\right],
\end{align*}
where the constants $C',C''$ depend on $a,b$. 
\end{proof}

\subsection{Proof of Theorem \ref{theo:resolution}}
\label{sub:proof_resolution}
In this section, we have $\mathcal Z = \R^\esd$. The notation $\|\cdot\|$ is  the euclidean norm either in $\R^D$ or in  $\R^\esd$. The constant $C$ may change from line to line.

\subsubsection{Preliminary results}

We consider an optimal $k$-points $\hat t := t(P_n^{\hat f})$ for the measure $P_n^{\hat f}$.  Let us introduce  the  distance function  $d_{\hat t}$ to  a $k$-points $\hat t$  of $(\R^\esd)^k$ : for any $z \in \R^\esd$,
$$d_{\hat t}(z) = \min_{j = 1,\dots ,k} \|z - \hat t_j \|.$$ We also introduce the  random variable  $$ \Delta  = \sup_{i=1, \dots, n} d_{\hat t}(f(X_i)) .$$  Let $\widehat{\mathcal U}^\epsilon = \{\hat U_j^\epsilon  \}_{j = 1, \dots, k}$ be the $\epsilon$-thickening of the Voronoi partition associated to $\hat t$. 
We start with the following lemma:
\begin{lemma} \label{lem_res_Delta}
Under assumptions {\rm (H1)} and {\rm (H3)},
$$
\frac 12 \res(\widehat{\mathcal U}^\epsilon, \hfpl)   \leq    \Delta +  3 \| (f - \hat f)\vert_{\hX_n}  \|_\infty  + \omega(\delta_n)  + \varepsilon .
$$
\end{lemma}
 \begin{proof}
Let $j \in \{1,\dots,k\}$ and $z' \in \hat U_j^\epsilon $. There exists $z \in \im(\hfpl)$
belonging to $j$-th Voronoi cell associated to $\hat t$ such that $\|z-z' \| \leq \varepsilon$. Let $x \in  G _{\delta_n}$ such that $z = \hat f(x)$. The point $x$ belongs to an edge $ [ X_{i_1} , X_{i_2} ]$ of  $ G_{\delta_n}$. We have 
\begin{eqnarray*} 
\|z' - t_j\| & \leq & \|z  - t_j\| + \|z' - z\| \\
 & \leq & \inf_{\ell = 1, \dots, k} \|\hat f (x) - t_\ell \|  + \varepsilon \\
& \leq & d_{\hat t}\left(\hat f (X_{i_1})\right)   + \| \hat f (X_{i_1}) - \hat f (x) \| + \varepsilon \\
& \leq & d_{\hat t}\left( f (X_{i_1})\right) + 3 \| (f - \hat f)\vert_{\hX_n}  \|_\infty   +    \|  f (X_{i_1}) -  f (x) \| + \varepsilon  \\
& \leq & \Delta + 3 \| (f - \hat f)\vert_{\hX_n}  \|_\infty   +  \omega(\delta_n) + \varepsilon, 
\end{eqnarray*}
where we have used the fact that $d_{\hat t}$ is one Lipschitz for the third inequality and the fact that $\|x -  X_{i_1} \| \leq \delta_n$ for the last inequality. The lemma follows.
 \end{proof}

In the following, we use standard notation in the field of empirical processes: for some integrable function $h$ with respect to some measure $Q$, let $Q h = \int h(x) dQ(x)$. Let $P^f$ be the push forward measure of $P$ by $f$.
\begin{lemma} \label{lem:CCtrick}
Under assumptions {\rm (H2)} and {\rm (H5)}, the following inequality holds conditionally to $\hX_n$:
$$  \Delta  \leq   C \left( P^f d_{\hat t}^2 \right)^{\frac \gamma{b+2\gamma}} \,  \vee  \, \left(  P^f d_{\hat t}^2 \right)^{\frac 12}   $$
where the constant $C$ only depends on $a$, $b$, $c$ and $\gamma$.
\end{lemma}

\begin{proof}
Let $\hat z \in \hZ_n$ such that $ \Delta   =  d_{\hat t}(\hat z)$. The function $z \in \Z \rightarrow d_{\hat t}(z)$ is one Lipschitz, thus for any $z \in \Z$ such that $\| z -  \hat z  \| \leq  \frac{ \Delta}2  $, we have
$$ 
\left|  d_{\hat t}(\hat z)- d_{\hat t}(z)  \right|  \leq \frac{ \Delta}2.
$$
This gives the inclusion 
$$
B\left(\hat z, \frac{ \Delta}2 \right) \subseteq \left\{ z \in \Z \, : \,  d_{\hat t}(z)  \geq \frac{ \Delta}2 \right\}.
$$
Then, by the Markov Inequality for $P^f$, we obtain
\begin{eqnarray*}
P^f d_{\hat t} ^2 & \geq & \frac{ \Delta^2}4  P^f \left( \left\{ z \in \Z \, : \,  d_{\hat t}(z)   \geq \frac{ \Delta}2   \right\} \right)   \\
& \geq & a \frac{ \Delta^2}4  \left[\omega^{-1} \left(\frac{ \Delta}2\right)\right]^b \wedge \frac{ \Delta^2}4 \\
& \geq & a \frac 1{2^b 4 c^{b/\gamma}}    \Delta    ^{\frac{2 \gamma + b}\gamma} \wedge \frac{ \Delta^2}4,
\end{eqnarray*}
where we have used Lemma~\ref{lem:Pf_standard} for the second inequality and (H5) for the third inequality. 
\end{proof}

 \begin{lemma} \label{lem:Pf_standard}
Under assumptions {\rm (H1)}, {\rm (H2)} and {\rm (H3)}, for any $r \geq 0$ and any $z \in \im(f)$, the push forward distribution $P^f$ satisfies the inequality 
 $$ P^f(B(z,r)) \geq a \left[\omega^{-1}(r)\right]^b.$$
\end{lemma} 
\begin{proof}
For any $r \geq 0$ and any $z=f(x) \in \im(f)$, by definition of the push forward measure $P^f$,
\begin{eqnarray*}
\int_\Z \mathds{1}_{B(f(x), r )} (z') d P^f (z') & \geq & \int_\X \mathds{1}_{B(f(x),r)} (f(x')) d P (x') \\
& \geq & \int_{B(x, \omega^{-1}(r))} \mathds{1}_{B(f(x),r)} (f(x')) d P (x') \\
& \geq & P \left(B(x, \omega^{-1}(r) \right) \\
& \geq & a  \left( \omega^{-1}(r) \right)^b.
\end{eqnarray*}
where we have used for the second inequality the fact that $\omega ( \omega^{-1}  (u))  = u$, because $\omega$ is continuous. 
\end{proof} 
Let $t^\star = t(P^f)$ be an optimal $k$ points for the measure $P^f$.  
\begin{lemma} \label{lem_Pf_t_star}
Under assumptions {\rm (H1)}, {\rm (H2)} and {\rm (H5)},
 $$P^f d_{t^\star}^2  \leq C k^{-\frac{2 \gamma}b}$$
where $C$ only depends on $a$,  $b$, $c$ and $\gamma$.
 \end{lemma} 
\begin{proof}
From Lemma~\ref{lem:Pf_standard}, it can be easily derived that  the $\delta$-covering number of the support of $P^f$ is  upper bounded by $C \delta^{-b/\gamma}$ where $C$ only depends on $a$, $c$, $b$ and $\gamma$ (see for instance the proof of Lemma 10 in \cite{JMLR:v16:chazal15a}). In other words, the minimum radius $\bar \delta$ to cover the support of $P^f$ with $k$ balls 
is upper bounded by $C k^{-\gamma / b}$. There exists a family of $k$ balls of radius $\bar \delta$ : $B(\bar  t_{j_1},\bar \delta ), \dots ,B(\bar  t_{j_k},\bar \delta )$ which is a cover of the support of $P^f$. We also define the function $\bar j : z \in \Z \mapsto \{1, \dots, k\}  $ which returns the index $j$ of  the (or one of the) closest center $\bar t_j$ to any point $z$ of the support of $P^f$.  Consequently,
\begin{eqnarray}
P^f d_{t^\star}^2 & \leq &  P^f d_{\bar t}^2  \notag \\
& \leq &    \E \left( \| Z -  \bar t_{\bar j(Z)} \|^2   \right) \notag \\
& \leq & \E \left[  \E \left( \left.\| Z -   \bar t_{\bar j(Z)} \|^2    \mathds{1}_{ \| Z -   \bar t_{\bar j(Z)} \| \leq \bar \delta } \right| j(Z) \right)\right] \label{ineq:JZ}
\end{eqnarray}
Conditionally to $\bar j(Z) =j$, one has
\begin{eqnarray*}
\E \left( \| Z - \bar t_j \|^2    \mathds{1}_{ \| Z - \bar t_j \| \leq \bar \delta } \right) 
& = & \int_{0}^{\bar \delta ^2}  P \left( \| Z - \bar t_j \|^2 > u \right)   du    \\
& = &  \int_{0}^{\bar \delta ^2} \left\{ 1 -  P \left( \| Z - \bar t_j \|^2  \leq u \right)  \right\} du  \\
& \leq & C k^{-\frac{2 \gamma}b} .
\end{eqnarray*}
We conclude by integrating this bound in Inequality~\eqref{ineq:JZ}.
\end{proof}

\subsubsection{Main part of the proof of Theorem \ref{theo:resolution}} 
 
From Lemmas \ref{lem:EspOmega}, \ref{lem_res_Delta}, and \ref{lem:CCtrick}, and by Jensen's Inequality, we find that an upper bound is
\begin{equation} \label{ineqresstart}
   C  \left\{
   \left[\E \left( P^f d_{\hat t}^2 \right) \right]^{\frac \gamma{b+2\gamma}}   + \left[   \E \left( P^f d_{\hat t }^2\right) \right]^{\frac12} 
    +   \E \| (f - \hat f)\vert_{\hX_n}  \|_\infty  +   \omega\left(\frac{{\rm log}(n)^{(2+\beta)/b}}{n^{1/b}}\right) \right\} + \varepsilon
\end{equation}
where $C$ only depends on $a$, $b$, $c$, $\gamma$ and on the geometric parameters of $\X$.
Next we need to upper bound the expectation of $P^f d_{\hat t}^2 $. Let $P_n^f$ be the push forward of $P_n$ by $f$, that is the empirical distribution corresponding to the $Z_i$'s. We start with the following standard decomposition:
\begin{eqnarray*} 
0 \leq P^f d_{\hat t}^2  - P^f d_{t^\star}^2  &=& (P^f - P_n^f) d_{\hat t}^2 +   P_n^f   d_{\hat t}^2 - P^f d_{t^\star}^2 \\
&\leq& (P^f - P_n^f) d_{\hat t}^2  +   (P_n^f - P^f)  d_{t^\star}^2 
\end{eqnarray*}  
where $t^\star = t(P^f)$ is an optimal $k$ points for the measure $P^f$. Note that $ t^\star  \in  ( B(0,\|f\|_\infty))^k$ and  that $ \hat t \in  (B(0,\|f\|_\infty))^k$ almost surely.  Thus,
\begin{equation}
\E P^f d_{\hat t}^2  \leq  P^f d_{t^\star}^2    +    2 \E  \sup_{t \in (B(0,\|f\|_\infty))^k}  \left|(P^f - P_n^f) d_{t}^2   \right|  \label{decomptoempproc}
\end{equation}
where the expectation is under the distribution of $\hZ_n$.

  \begin{proposition} \label{prop:procem}
  The following inequality holds:
 $$ \E  \sup_{t \in (B(0,\|f\|_\infty))^k}  \left|(P^f - P_n^f) d_{t}^2   \right|  \leq  \frac{C\|f\|_{\infty}^2}{\sqrt n} \sqrt { k (\esd+2)}  $$
 where $C$ is an absolute constant.
  \end{proposition}
 
 \begin{proof}
We introduce the functional spaces
 $$\mathcal G_1 =\left\{z  \mapsto  \|  z  -  t_1 \|^2 \mathds{1}_{B(0,\|f\|_\infty)}(z) \, :  \, t_1 \in  B\left(0,\|f\|_\infty\right)   \right\} $$ 
 and 
 \begin{eqnarray*}  
 \mathcal G 
 &=& \left\{ z   \mapsto d_{t}^2(z) \mathds{1}_{B(0,\|f\|_\infty)}(z)  \, :  \, t \in \left(B\left(0,\|f\|_\infty\right)\right)^k  \right\}  \\
 &=& \left\{ z  \mapsto \min_{j=1,\dots, k} l_j(z)  \, : \, l_j \in \mathcal G_1 \right\} .
 \end{eqnarray*}  
 Note that $ 0 \leq g \leq   4\|f\|_\infty^2$ for any $g \in \mathcal G$. According to Theorem~\ref{theo:Dudley} and Lemma~\ref{lem_Brecheteau33}, 
\begin{eqnarray}
\E \left[\sup_{g \in \mathcal G} \left| (P -P_n) g \right|    \right]  
&\leq&
96 \frac{ \|f\|_\infty^2}{\sqrt n} \E \left[
\int_{0}^{\frac12} \sqrt{ \log \left( N'_{\| \cdot \|} \left( \frac u2 , \frac{( \mathcal G \cup - \mathcal G) ( Z_1^n) }{4\|f\|_\infty^2 \sqrt n} \right)\right)} du \right] \notag  \\
&\leq& 96 \frac{ \|f\|_\infty^2}{\sqrt n} \E \left[
\int_{0}^{\frac12} \sqrt{ \log \left( 2 N'_{\| \cdot \|} \left( \frac u2 , \frac{  \mathcal G  ( Z_1^n) }{ \|f\|_\infty^2 \sqrt n} \right) \right)} du \right] \notag \\
&\leq& 96 \frac{ 4\|f\|_\infty^2}{\sqrt n} \E \left[
\int_{0}^{\frac12} \sqrt{ \log 2  + k \log \left(   N'_{\| \cdot \|} \left( \frac u2 , \frac{  \mathcal G_1( Z_1^n) }{ 4\|f\|_\infty^2 \sqrt n} \right)    \right)} du \right]. \label{dud}
\end{eqnarray}
According to Lemma~\ref{lem_Brecheteau33}, 
$$N'_{\| \cdot \|} \left( \frac u2 , \frac{  \mathcal G_1( Z_1^n) }{4\|f\|_\infty^2 \sqrt n} \right) \leq 
N'_{\| \cdot \|} \left( \frac u4 ,    \mathcal G_2( Z_1^n)   \right) 
N'_{\| \cdot \|} \left( \frac u4 ,    \mathcal G_3( Z_1^n)   \right) .
$$
where $$ \mathcal G_2 =\left\{z  \mapsto \frac{\| t_1 \|^2}{4\|f\|_\infty^2 \sqrt n} \mathds{1}_{B(0,\|f\|_\infty)}(z) \, :  \, t_1 \in  B(0,\|f\|_\infty)   \right\} $$ 
and 
$$ \mathcal G_3 =\left\{ z  \mapsto \frac{\langle z ,  t_1\rangle}{2\|f\|_\infty^2 \sqrt n}  \mathds{1}_{B(0,\|f\|_\infty)}(z)  :    \,   t_1 \in  B(0,\|f\|_\infty)   \right\} .$$
Note that 
$ \mathcal G_2 \subset \mathcal G_4 =  \left\{ z  \mapsto \frac{u}{\sqrt n} \mathds{1}_{B(0,\|f\|_\infty)}(z)     \, :  \, u \in  [0,1/4]   \right\}$
and thus $$N'_{\| \cdot \|} \left( \frac u4 , \mathcal G_2 ( Z_1^n) \right) \leq N'_{\| \cdot \|} \left( \frac u4 , \mathcal G_4( Z_1^n) \right) \leq  \frac 2 \delta.$$ 
Next, according to Theorem~\ref{theo_mendelson2003entropy} and Lemma~\ref{lem_CClem37}, $N'_{\| \cdot \|} \left( \frac u4 , \frac{  \mathcal G_3( Z_1^n) }{4\|f\|_\infty^2 \sqrt n} \right)\leq  \left( \frac{8}{\delta} \right)^{c_1   (\esd+2)  } .$  As a consequence,
$$ \log N'_{\| \cdot \|} \left( \frac u2 , \frac{  \mathcal G_1( Z_1^n) }{4\|f\|_\infty^2 \sqrt n} \right) \leq
(1+c_1(\esd+2)) \log \frac 8\delta 
$$
and we conclude with~\eqref{dud}.
 \end{proof}

{\bf End of the proof of Theorem \ref{theo:resolution}.} According to Inequalities~\eqref{ineqresstart} and \eqref{decomptoempproc}, Proposition~\ref{prop:procem} and Lemma~\ref{lem_Pf_t_star}, and using the fact that $u \mapsto u^\zeta$ is a sub additive function for $\zeta \in (0,1)$, it follows that one has the upper bound
\begin{eqnarray*}
   & & C  \left[
   \left[k^{-\frac{2 \gamma}b} +    \sqrt {\frac {k(\esd+2)}{n}} \|f\|_{\infty}^2 \right]^{\frac \gamma{b+2\gamma}}   + \left[   k^{-\frac{2 \gamma}b} +  \sqrt {\frac {k(\esd+2)}{n}} \|f\|_{\infty}^2\right]^{\frac12}
   \right]    \\
   & & + 3 \E \| (f - \hat f)\vert_{\hX_n}  \|_\infty  + C \omega\left(\frac{{\rm log}(n)^{(2+\beta)/b}}{n^{1/b}}\right)   + \varepsilon \\
   & \leq  & C  \left[ k^{-\frac{2 \gamma^2}{b^2+ 2 \gamma b}} +  \left( \frac {k(\esd+2)}{n} \right) ^{\frac \gamma{2b+4\gamma}} + \left(  \frac {k(\esd+2)}{n} \right) ^{\frac 14}   \right] + 3 \E \| (f - \hat f)\vert_{\hX_n}  \|_\infty + C \left(  \frac{ {\rm log}(n)^{2+\beta} } n\right)^{\gamma/b}        + \varepsilon
\end{eqnarray*}
where the constants $C$ depends on $a$, $b$, $c$, $\gamma$, $\|f\|_{\infty}$ and on the geometric parameters of $\X$. 
For $n \geq k(\esd+2)$, this upper bound can be rewritten as
\begin{eqnarray*}
 & & C  \left[ k^{-\frac{2 \gamma^2}{b^2+ 2 \gamma b}} +  \left( \frac {k \esd }{n} \right) ^{\frac \gamma{2b+4\gamma}}    \right] + 6 \E \| (f - \hat f)\vert_{\hX_n}  \|_\infty +       2 \varepsilon    .
\end{eqnarray*}
This concludes the proof of Theorem \ref{theo:resolution}.

\subsubsection{Dudley's entropy integral and tools for covering numbers}

For ease of reading, several result about the Dudley's entropy integral and covering numbers are recalled in this section. Our presentation is inspired from Section B.1 in \cite{brecheteau2020k}.

\medskip

Let $\mathcal G$ and $\mathcal G'$ be two countable families of functions $g:\R^\esd \rightarrow \R$. The set $ \mathcal G  ( Z_1^n)$ is the set $$\mathcal G  ( Z_1^n)=\{ (g(Z_1), \dots, g(Z_n)) \, : \, g \in \mathcal G  \}.$$ 
For $S \subset \R^\esd$, let $N'_{\|\cdot \|}(\delta,S)$ denotes the $\delta$ covering number of $S$ with respect to the euclidean norm $\|\cdot\|$ in $\R^\esd$. 

Let $ Z_1,\dots, Z_n$ sampled according to $P$, which is a distribution on $\R^\esd$, and let $P_n$ be the corresponding empirical measure. The next result is a particular instance of the so-called Dudley's integral.

 \begin{theorem}{\cite[Corollary 13.2]{boucheron2013concentration}} \label{theo:Dudley}
Assume that $\mathcal G$ contains the null function and that $g \leq R$ for any $g \in \mathcal G$. Then,
$$ \E \left[\sup_{g \in \mathcal G} \left| (P -P_n) g \right|    \right]  \leq
24 \frac{R}{\sqrt n} \E \left[
\int_{0}^{\frac12} \sqrt{ \log \left( N'_{\| \cdot \|} \left( \frac u2 , \frac{( \mathcal G \cup - \mathcal G) ( Z_1^n) }{R \sqrt n} \right)\right)} du
\right].$$
\end{theorem}

\begin{lemma}{\cite[ Lemma 33]{brecheteau2020k}}\label{lem_Brecheteau33}
Let $\delta >0$. Let $\mathcal G_{(k)} = \{ \min_{j = 1,\dots, k} g_j \, : \, g_j \in \mathcal G\}  $ and $\mathcal G +\mathcal G'  = \{ g  + g' \, , \, g  \in \mathcal G , g' \in \mathcal G'\}$. The following inequalities hold:
\begin{itemize}
    \item   $  N'_{\| \cdot \|} \left(  \delta  ,   (\mathcal G \cup - \mathcal G  )   ( z_1^n) \right) \leq  2   N'_{\| \cdot \|} \left(   \delta   ,  \mathcal G     ( z_1^n) \right)    $ 
\item $  N'_{\| \cdot \|} \left(  \delta  ,   \mathcal G_{(k)}    ( z_1^n) \right) \leq  \left( N'_{\| \cdot \|} \left(   \delta   ,  \mathcal G     ( z_1^n) \right) \right)^k $ 
\item $ 
$$  N'_{\| \cdot \|} \left(  2\delta  ,    (\mathcal G +  \mathcal G') ( z_1^n) \right) \leq    N'_{\| \cdot \|} \left(   \delta   ,  \mathcal G     ( z_1^n) \right)  N'_{\| \cdot \|} \left(   \delta   ,  \mathcal G'     ( z_1^n) \right)  .$
\end{itemize}
\end{lemma}

It is possible to control the covering number of a set $\mathcal G( z_1^n)$ by  the $\delta$-fat-dimension (also called $\delta$-shattering dimension) of the family $\mathcal G$.

\begin{definition} Let $\delta > 0 $. 
\begin{itemize}
\item A set $\{z_1, \dots, z_m\} \subset \R^\esd$ is said to be $\delta$-shattered by $\mathcal G$ if there exists  $(u_1,\dots,u_m) \in \R^m$ such that for all $(\varepsilon_1, \dots, \varepsilon_m) \in  \{-1, +1\}^m $, there exists $g \in \mathcal G$ such that:
$$ \forall i \in \{1,\dots,m \}, \,  \varepsilon_i(g(z_i) - u_i) \geq \delta.$$
\item The $\delta$ fat-dimension of $\mathcal G$, $\fat_\delta(\mathcal G)$, is the size of the largest set in $\mathbb R^\esd$ that is $\delta$-shattered by $\mathcal G$.
 \end{itemize}
\end{definition}

 \begin{theorem}{\cite[Theorem 1]{mendelson2003entropy}} \label{theo_mendelson2003entropy}
Assume that class of functions $\mathcal G$ is bounded by 1. There exists absolute constants $c_1$ and $c_2$ such that for all $ z_1^n \in (\R^\esd)^n$ and all $\delta \in (0,1)$,
$$ N'_{\| \cdot \|} \left(  \delta  ,  \frac 1{\sqrt n} \mathcal G ( z_1^n) \right)  \leq \left( \frac{2}{\delta} \right)^{c_1 \fat_{c_2 \gamma}(\mathcal G)}. $$ 
\end{theorem}

 \begin{lemma}{\cite[Lemma 37]{brecheteau2020k}} \label{lem_CClem37}
 Let $R >0$ and $\mathcal H = \{ z \mapsto \frac 1R \mathds{1}_{B(0,R)}(z) \langle z,v \rangle \, : \, v \in S(0,1)\} $
 where $S(0,1)$ is the unit sphere of $\R^\esd$. Then, for any $\delta >0 $,
 $$ \fat_{\delta}(\mathcal H) \leq \esd +2.$$ 
\end{lemma}
\end{document}